\RequirePackage{snapshot}

\newif\ifsubsections
\subsectionstrue
\documentclass[]{pcmi}

\usepackage[sc]{mathpazo}          % Palatino with smallcaps as text font
\usepackage{eulervm}               % Euler math
\usepackage[scaled=0.86]{berasans} % Bera for san serif family
\usepackage[scaled=1]{inconsolata} % Inconsolata for fixed width
\usepackage[T1]{fontenc}

\usepackage[%
	protrusion=true,
	expansion=false,
	auto=false
	]{microtype}

\usepackage{xcolor}
\usepackage{graphicx}
\graphicspath{{./figures/}}

\ifdraft
	\definecolor{linkred}{rgb}{0.7,0.2,0.2}
	\definecolor{linkblue}{rgb}{0,0.2,0.6}
\else
	\definecolor{linkred}{rgb}{0.0,0.0,0.0}
	\definecolor{linkblue}{rgb}{0,0.0,0.0}
\fi

\usepackage{amsfonts, amssymb, amsmath, enumerate, amsthm}
\usepackage{subcaption}
\usepackage[font=small]{caption}
\usepackage[english]{babel}

\PassOptionsToPackage{hyphens}{url} 
\usepackage[
    setpagesize=false,
    pagebackref,
	pdfpagelabels=false,
    pdfstartview={FitH 1000},
    bookmarksnumbered=false,
    linktoc=all,
    colorlinks=true,
    anchorcolor=black,
    menucolor=black,
    runcolor=black,
    filecolor=black,
    linkcolor=linkblue,%IM black if not in draft mode
	citecolor=linkblue,%IM black if not in draft mode
	urlcolor=linkred,%IM black if not in draft mode
]{hyperref}%IM
\usepackage[backrefs,msc-links,nobysame]{amsrefs}

\customizeamsrefs 

\theoremstyle{plain}
\newtheorem{theorem}[equation]{Theorem}
\newtheorem{proposition}[equation]{Proposition}
\newtheorem{corollary}[equation]{Corollary}
\newtheorem{lemma}[equation]{Lemma}

\newtheorem{definition}[equation]{Definition}

\theoremstyle{remark}
\newtheorem{remark}[equation]{Remark}

\newtheorem{example}[equation]{Example}

\DeclareMathOperator{\E}{\mathbb{E}}
\DeclareMathOperator{\Var}{Var}

\DeclareMathOperator{\Ber}{Ber}

\DeclareMathOperator{\rad}{rad}
\DeclareMathOperator{\rank}{rank}
\DeclareMathOperator{\tr}{tr}

\renewcommand{\Pr}[2][]{\mathbb{P}_{#1} \left\{ #2 \rule{0mm}{3mm}\right\}}
\newcommand{\ip}[2]{\left\langle#1,#2\right\rangle}

\def \P {\mathbb{P}}
\def \R {\mathbb{R}}

\def \XX {\mathcal{X}}

\def \b {\beta}

\def \e {\varepsilon}
\def \d {\delta}
\def \l {\lambda}
\def \s {\sigma}

\def \tran {\mathsf{T}}

\def \psione {{\psi_1}}
\def \psitwo {{\psi_2}}

\long\def\replace#1{#1}

\replace{%

}

\begin{document}

\title{Four lectures on probabilistic methods for data science}
\author{Roman Vershynin}
\address{Department of Mathematics, University of Michigan, 530 Church Street, Ann Arbor, MI 48109, U.S.A.}
\email{romanv@umich.edu}
\date{\today}
\thanks{Partially supported by NSF Grant DMS 1265782 and U.S. Air Force Grant FA9550-14-1-0009.}

\begin{abstract}
 Methods of high-dimensional probability play a central role in applications for statistics, signal processing,
 theoretical computer science and related fields. These lectures present a sample of particularly useful tools
 of high-dimensional probability, focusing on the classical and matrix Bernstein's inequality and the 
 uniform matrix deviation inequality. We illustrate these tools with applications 
 for dimension reduction, network analysis, covariance estimation, matrix completion and sparse signal recovery.
 The lectures are geared towards beginning graduate students who have taken a rigorous course in probability 
 but may not have any experience in data science applications. 
\end{abstract}  

\maketitle

\tableofcontents

%%%%%%%%%%%%%%%%%%%%%%%%%%%%%%%%%%%%%%%%%%%%%%%%%%%%%%%%%%%%%%%%%%%%%
%    
%    Your lecture notes replace the remainder of this document.
%    
%%%%%%%%%%%%%%%%%%%%%%%%%%%%%%%%%%%%%%%%%%%%%%%%%%%%%%%%%%%%%%%%%%%%%

% Authors: insert the body of your lectures here

%%%%%%%%%%%%%%%
%\chapter*{Preface}
%%%%%%%%%%%%%%%
%
%Lorem \index{lorem} ipsum dolor sit amet, consectetur adipiscing elit. Duis risus ante, auctor et pulvinar non, posuere ac lacus. Praesent egestas nisi id metus rhoncus ac lobortis sem hendrerit. Etiam et sapien eget lectus interdum posuere sit amet ac urna.
%
%\section*{Un-numbered sample section}
%%=============================
%
%Lorem ipsum dolor sit amet, consectetur adipiscing elit. Duis risus ante, auctor et pulvinar non, posuere ac lacus. Praesent egestas nisi id metus rhoncus ac lobortis sem hendrerit. Etiam et sapien eget lectus interdum posuere sit amet ac urna. Aliquam pellentesque imperdiet erat, eget consectetur felis malesuada quis. Pellentesque sollicitudin, odio sed dapibus eleifend, magna sem luctus turpis.
%

%%%%%%%%%%%%%%%%%%%
\section{Lecture 1: Concentration of sums of independent random variables}	
%%%%%%%%%%%%%%%%%%%

These lectures present a sample of modern methods of high dimensional probability
and illustrate these methods with applications in data science. This sample is not comprehensive
by any means, but it could serve as a point of entry into a branch of modern probability
that is motivated by a variety of data-related problems.

To get the most out of these lectures,
you should have taken a graduate course in probability,
have a good command of linear algebra (including the singular value decomposition)
and be familiar with very basic concepts of functional analysis (familiarity with $L^p$ norms 
should be enough).

All of the material of these lectures is covered more systematically, at a slower pace, 
and with a wider range of applications, in my forthcoming textbook \cite{V-textbook}. 
You may also be interested in two similar tutorials: \cite{V-RMT-tutorial} 
is focused on random matrices, and a more advanced text \cite{V-estimation-tutorial} 
discusses high-dimensional inference problems.

It should be possible to use these lectures for a self-study or group study. 
You will find here many places where you are invited to do some work (marked in the text e.g. by ``check this!''),
and you are encouraged to do it to get a better grasp of the material. 
Each lecture ends with a section called ``Notes'' where you will find bibliographic references of 
the results just discussed, as well asvarious improvements and extensions.  

We are now ready to start. 

\medskip

Probabilistic reasoning has a major impact on modern data science. 
There are roughly two ways in which this happens.  

\begin{itemize}
  \item {\em Radnomized algorithms}, which perform some operations at random, have long 
    been developed in computer science and remain very popular. 
    Randomized algorithms are among the most effective methods 
    -- and sometimes the only known ones -- for many data problems. 
  \item {\em Random models of data} form the usual premise of statistical analysis. 
    Even when the data at hand is deterministic, it is often helpful to think of it as a random sample 
    drawn from some unknown distribution (``population'').
\end{itemize}

In these lectures, we will encounter both randomized algorithms and random models of data.

\subsection{Sub-gaussian distributions}			\label{s: sub-gaussian}
%----------------

Before we start discussing probabilistic methods, we will introduce an important class
of probability distributions that forms a natural ``habitat'' for random variables in many 
theoretical and applied problems. These are {\em sub-gaussian} distributions.
As the name suggests, we will be looking at an extension of the most 
fundamental distribution in probability theory -- the gaussian, or normal, distribution $N(\mu,\s)$.

It is a good exercise to check that the standard normal random variable $X \sim N(0,1)$ satisfies
the following basic properties: 

\begin{description}
  \item [Tails] $\Pr{|X| \ge t}  \le 2 \exp(-t^2/2)$ for all $t \ge 0$.
  \item [Moments] $\|X\|_p := (\E |X|^p)^{1/p} = O(\sqrt{p})$ as $p \to \infty$.
  \item [MGF of square]\footnote{MGF stands for moment generation function.} $\E \exp(cX^2) \le 2$ for some $c > 0$.
  \item [MGF] $\E \exp(\l X) = \exp(\l^2)$ for all $\l \in \R$.
\end{description}

All these properties tell the same story from four different perspectives. 
It is not very difficult to show (although we will not do it here) 
that for any random variable $X$, not necessarily Gaussian, these four properties are essentially equivalent. 

\begin{proposition}[Sub-gaussian properties]		\label{prop: sub-gaussian}
  For a random variable $X$, the following properties are equivalent.\footnote{The parameters
  $K_i > 0$ appearing in these properties can be different. However, they may differ from each other 
  by at most an absolute constant factor. 
  This means that there exists an absolute constant $C$ such that property $1$
    implies property $2$ with parameter $K_2 \le C K_1$, and similarly for every other pair or properties.}
  \begin{description}
    \item [Tails] $\Pr{ |X| \ge t } \le 2 \exp(-t^2/K_1^2)$ for all $t \ge 0$.
    \item [Moments] $\|X\|_p \le K_2 \sqrt{p}$ for all $p \ge 1$.
    \item [MGF of square] $\E \exp(X^2/K_3^2) \le 2$.
  \end{description}
  Moreover, if $\E X = 0$ then these properties are also equivalent to the following one:
  \begin{description} 
    \item [MGF] $\E \exp(\l X) \le \exp(\l^2 K_4^2)$ for all $\l \in \R$.
  \end{description}
\end{proposition}

Random variables that satisfy one of the first three properties (and thus all of them) 
are called {\em sub-gaussian}. The best $K_3$ is called the {\em sub-gaussian norm} of $X$, 
and is usually denoted $\|X\|_\psitwo$, that is
$$
\|X\|_\psitwo := \inf \left\{ t>0 :\; \E \exp(X^2/t^2) \le 2 \right\}.
$$
One can check that $\|\cdot\|_\psitwo$ indeed defines a norm; it is an example of the general concept 
of the {\em Orlicz norm}. 
Proposition~\ref{prop: sub-gaussian} states that the numbers $K_i$ in all four properties 
are equivalent to $\|X\|_\psitwo$ up to absolute constant factors. 

\begin{example}
  As we already noted, the standard normal random variable $X \sim N(0,1)$ is sub-gaussian. 
  Similarly, arbitrary {\em normal} random variables $X \sim N(\mu, \s)$ are sub-gaussian. 
  Another example is a {\em Bernoulli} random variable $X$ that takes values $0$ and $1$ 
  with probabilities $1/2$ each. More generally, any {\em bounded} random variable $X$
  is sub-gaussian. On the contrary, Poisson, exponential, Pareto and Cauchy distributions
  are not sub-gaussian. (Verify all these claims; this is not difficult.)
\end{example}

\subsection{Hoeffding's inequality}
%-----------

You may remember from a basic course in probability that the normal distribution $N(\mu,\s)$ 
has a remarkable property: the sum of independent 
normal random variables is also normal.
Here is a version of this property for sub-gaussian distributions.

\begin{proposition}[Sums of sub-gaussians] \label{prop: sum of sub-gaussians}
  Let $X_1,\ldots,X_N$ be independent, mean zero, sub-gaussian random variables.
  Then $\sum_{i=1}^N X_i$ is a sub-gaussian, and
  $$
  \Big\| \sum_{i=1}^N X_i \Big\|_\psitwo^2 
  \le C \sum_{i=1}^N \|X_i\|_\psitwo^2
  $$
  where $C$ is an absolute constant.\footnote{In the future, we will always denote positive absolute constants 
    by $C$, $c$, $C_1$, etc. These numbers do not depend on anything. In most cases, one can get 
    good bounds on these constants from the proof, but the optimal constants for each result are rarely known.} 
\end{proposition} 

\begin{proof}
Let us bound the moment generating function of the sum for any $\l \in \R$:
\begin{align*}
\E \exp \big( \l \sum_{i=1}^N X_i \big)
&= \prod_{i=1}^N \E \exp(\l X_i)		\quad \text{(using independence)}  \\
&\le \prod_{i=1}^N \exp(C \l^2 \|X_i\|_\psitwo^2)	
	\quad \text{(by last property in Proposition~\ref{prop: sub-gaussian})}\\
&= \exp(\l^2 K^2) \quad \text{where } K^2 := C \sum_{i=1}^N \|X_i\|_\psitwo^2.
\end{align*}
Using again the last property in Proposition~\ref{prop: sub-gaussian}, we conclude that the sum
$S = \sum_{i=1}^N X_i$ is sub-gaussian, and $\|S\|_\psitwo \le C_1 K$
where $C_1$ is an absolute constant. 
The proof is complete.
\end{proof}

Let us rewrite Proposition~\ref{prop: sum of sub-gaussians} in a form 
that is often more useful in applications, namely as a {\em concentration inequality}. To do this, 
we simply use the first property in Proposition~\ref{prop: sub-gaussian} for the sum $\sum_{i=1}^N X_i$.
We immediately get the following. 

\begin{theorem}[General Hoeffding's inequality]		  	\label{thm: Hoeffding}
  Let $X_1,\ldots,X_N$ be independent, mean zero, sub-gaussian random variables.
  Then, for every $t \ge 0$ we have
  $$
  \P \Big\{ \Big| \sum_{i=1}^N X_i \Big| \ge t \Big\} 
  \le 2 \exp \Big( -\frac{ct^2}{\sum_{i=1}^N \|X_i\|_\psitwo^2} \Big).
  $$   
\end{theorem}

Hoeffding's inequality controls how far and with what probability
a sum of independent random variables can deviate from its mean, which is zero.

\subsection{Sub-exponential distributions}
%---------------

Sub-gaussian distributions form a sufficiently wide class of distributions. 
Many results in probability and data science are proved nowadays 
for sub-gaussian random variables. Still, as we noted, 
there are some natural random variables that are not sub-gaussian. 
For example, the square $X^2$ of a normal random variable $X \sim N(0,1)$ is not sub-gaussian. (Check!)
To cover examples like this, we will introduce the similar but weaker notion of {\em sub-exponential distributions}. 

\begin{proposition}[Sub-exponential properties]		\label{prop: sub-exponential}
  For a random variable $X$, the following properties are equivalent, in the same sense as 
  in Proposition~\ref{prop: sub-gaussian}.
  \begin{description}
    \item [Tails] $\Pr{ |X| \ge t } \le 2 \exp(-t/K_1)$ for all $t \ge 0$.
    \item [Moments] $\|X\|_p \le K_2 p$ for all $p \ge 1$.
    \item [MGF of the square] $\E \exp(|X|/K_3) \le 2$.
  \end{description}
  Moreover, if $\E X = 0$ then these properties imply the following one:
  \begin{description} 
    \item [MGF] $\E \exp(\l X) \le \exp(\l^2 K_4^2)$ for $|\l| \le 1/K_4$.
  \end{description}
\end{proposition}

Just like we did for sub-gaussian distributions, we call the best $K_3$ the {\em sub-exponential norm} of $X$ 
and denote it by $\|X\|_\psione$, that is
$$
\|X\|_\psione := \inf \left\{ t>0 :\; \E \exp(|X|/t) \le 2 \right\}.
$$

All sub-exponential random variables are squares of sub-gaussian random variables. 
Indeed, inspecting the definitions you will quickly see that
\begin{equation}         \label{eq: psione psitwo}
\|X^2\|_\psione = \|X\|_\psitwo^2.
\end{equation}
(Check!)

\subsection{Bernstein's inequality}
%----------

A version of Hoeffding's inequality for sub-exponential random variables is called Bernstein's inequality. 
You may naturally expect to see a sub-exponential tail bound in this result. 
So it may come as a surprise that Bernstein's inequality actually has a mixture of {\em two} tails 
-- sub-gaussian and sub-exponential. 
Let us state and prove the inequality first, and then we will comment on the mixture of the two tails.

\begin{theorem}[Bernstein's inequality]	\label{thm: Bernstein}
  Let $X_1,\ldots,X_N$ be independent, mean zero, sub-exponential random variables.
  Then, for every $t \ge 0$ we have
  $$
  \P \Big\{ \Big| \sum_{i=1}^N X_i \Big| \ge t \Big\} 
  \le 2 \exp \Big[ -c \min \Big( \frac{t^2}{\sum_{i=1}^N \|X_i\|_\psione^2}, \; \frac{t}{\max_i \|X_i\|_\psione} \Big) \Big].
  $$
\end{theorem}

\begin{proof}
For simplicity, we will assume that $K=1$ and only prove the one-sided bound (without absolute value);
the general case is not much harder.
Our approach will be based on bounding the {\em moment generating function} of the sum $S := \sum_{i=1}^N X_i$. 
To see how MGF can be helpful here, choose $\l \ge 0$ and use Markov's inequality to get
\begin{equation}         \label{eq: tail begin}
\Pr{ S \ge t } 
= \Pr{ \exp(\l S) \ge \exp(\l t) }
\le e^{-\l t} \E \exp(\l S).
\end{equation}
Recall that $S= \sum_{i=1}^N X_i$ and use independence to express the right side of \eqref{eq: tail begin} as
$$
e^{-\l t} \prod_{i=1}^N \E \exp(\l X_i).
$$
(Check!)
It remains to bound the MGF of each term $X_i$, and this is a much simpler task. 
If we choose $\l$ small enough so that
\begin{equation}         \label{eq: lambda small}
0 < \l \le \frac{c}{\max_i \|X_i\|_\psione},
\end{equation}
then we can use the last property in Proposition~\ref{prop: sub-exponential} to get
$$
\E \exp(\l X_i) \le \exp \big( C \l^2 \|X_i\|_\psione^2 \big).
$$
Substitute into \eqref{eq: tail begin} and conclude that
$$
\P \{ S \ge t \}  
\le \exp \left( -\l t + C \l^2 \s^2 \right)
$$
where $\s^2 = \sum_{i=1}^N \|X_i\|_\psione^2$.
The left side does not depend on $\l$ while the right side does. 
So we can choose $\l$ that minimizes the right side subject to the constraint \eqref{eq: lambda small}.
When this is done carefully, we obtain the tail bound stated in Bernstein's inequality. (Do this!)
\end{proof}

Now, why does Bernstein's inequality have a mixture of two tails? 
The sub-exponential tail should of course be there. Indeed, even if the entire sum consisted of a single term $X_i$, 
the best bound we could hope for would be of the form $\exp(-ct/\|X_i\|_\psione)$. 
The sub-gaussian term could be explained by the central limit theorem, which states
that the sum should becomes approximately {\em normal} as the number of terms $N$ increases to infinity.

\begin{remark}[Bernstein's inequality for bounded random variables]
  Suppose the random variables $X_i$ are {\em uniformly bounded}, 
  which is a stronger assumption than being sub-gaussian. 
  Then there is a useful version of Bernstein's inequality, which unlike Theorem~\ref{thm: Bernstein}
  is sensitive to the {\em variances} of $X_i$'s. It states that if $K>0$ is such that 
  $|X_i| \le K$ almost surely for all $i$, then, for every $t \ge 0$, we have
  \begin{equation}         \label{eq: Bernstein bounded}
  \P \Big\{ \Big| \sum_{i=1}^N X_i \Big| \ge t \Big\} 
  \le 2 \exp \Big( -\frac{t^2/2}{\s^2 + C K t} \Big).
  \end{equation}
  Here $\s^2 = \sum_{i=1}^N \E X_i^2$ is the variance of the sum. 
  This version of Bernstein's inequality can be proved in essentially the same way as Theorem~\ref{thm: Bernstein}.
  We will not do it here, but a stronger Theorem~\ref{thm: matrix Bernstein}, 
  which is valid for matrix-valued random variables $X_i$, 
  will be proved in Lecture~\ref{s: Lecture 2}.
  
  To compare this with Theorem~\ref{thm: Bernstein}, note that $\s^2 + CKt \le 2\max(\s^2, CKt)$. 
  So we can state the probability bound \eqref{eq: Bernstein bounded} as 
  $$
  2 \exp \Big[ -c \min \Big( \frac{t^2}{\s^2}, \; \frac{t}{K} \Big) \Big].
  $$
  Just like before, here we also have a mixture 
  of two tails, sub-gaussian and sub-exponential. The sub-gaussian tail is a bit sharper than in 
  Theorem~\ref{thm: Bernstein}, since it depends on the {\em variances} rather than sub-gaussian norms
  of $X_i$. The sub-exponential tail, on the other hand, is weaker, since it depends on the sup-norms
  rather than the sub-exponential norms of $X_i$.
\end{remark}

\subsection{Sub-gaussian random vectors}		\label{s: sub-gaussian vectors}
%---------

The concept of sub-gaussian distributions can be extended to higher dimensions. 
Consider a random vector $X$ taking values in $\R^n$. 
We call $X$ a {\em sub-gaussian random vector} 
if all one-dimensional {\em marginals} of $X$, i.e., the random variables $\ip{X}{x}$ for $x \in \R^n$,
are sub-gaussian. The sub-gaussian norm of $X$
is defined as
$$
\|X\|_\psitwo := \sup_{x \in S^{n-1}} \|\ip{X}{x} \|_\psitwo
$$
where $S^{n-1}$ denotes the unit Euclidean sphere in $\R^n$.

\begin{example}
  Examples of sub-gaussian random distributions in $\R^n$ include the standard normal distribution 
  $N(0,I_n)$ (why?), the uniform distribution on the centered Euclidean sphere
  of radius $\sqrt{n}$, the uniform distribution on the cube $\{-1,1\}^n$, and many others. 
  The last example can be generalized: a random vector $X = (X_1,\ldots,X_n)$ with independent 
  and sub-gaussian coordinates is sub-gaussian, with $\|X\|_\psitwo \le C \max_i \|X_i\|_\psitwo$.   
\end{example}

\subsection{Johnson-Lindenstrauss Lemma}		\label{s: JL}
%------------	

Concentration inequalities like Hoeffding's and Bernstein's are successfully used in the analysis
of algorithms. Let us give one example for the problem of {\em dimension reduction}.
Suppose we have some data that is represented as a set of $N$ points in $\R^n$. 
(Think, for example, of $n$ gene expressions of $N$ patients.) 

We would like to compress the data by representing it in a lower dimensional space 
$\R^m$ instead of $\R^n$ with $m \ll n$. 
By how much can we reduce the dimension without loosing the important 
features of the data? 

The basic result in this direction is the Johnson-Lindenstrauss Lemma.
It states that a remarkably simple dimension reduction method works -- a {\em random linear map} 
from $\R^n$ to $\R^m$ with
$$
m \sim \log N,
$$
see Figure~\ref{fig: JL}.
The logarithmic function grows very slowly, so we can usually 
reduce the dimension dramatically.

What exactly is a random linear map? Several models are possible to use. 
Here we will model such a map using a {\em Gaussian random matrix} -- an $m \times n$  matrix $A$ with 
independent $N(0,1)$ entries. 
More generally, we can consider an $m \times n$ matrix $A$ whose rows are 
independent, mean zero, isotropic\footnote{A random vector $X \in \R^n$ is called {\em isotropic} if $\E X X^\tran = I_n$.}
and sub-gaussian random vectors in $\R^n$. For example, the entries of $A$ can be independent 
{\em Rademacher} entries -- those taking values $\pm 1$ with equal probabilities.

\begin{theorem}[Johnson-Lindenstrauss Lemma]			\label{thm: JL}
  Let $\XX$ be a set of $N$ points in $\R^n$ and $\e \in (0,1)$.
  Consider an $m \times n$ matrix $A$ whose rows are independent, mean zero, 
  isotropic and sub-gaussian random vectors in $\R^n$. 
  Rescale $A$ by defining the ``Gaussian random projection''\footnote{Strictly speaking, this $P$ is not 
    a projection since it maps $\R^n$ to a different space $\R^m$.}
  $$
  P := \frac{1}{\sqrt{m}} A. 
  $$
  Assume that 
  $$
  m \ge C \e^{-2} \log N,
  $$
  where $C$ is an appropriately large constant that depends only on the sub-gaussian norms of the vectors $X_i$.
  Then, with high probability (say, $0.99$), 
  the map $P$ preserves the distances between all points in $\XX$ with error $\e$, 
  that is 
  \begin{equation}         \label{eq: JL}
  (1-\e) \|x-y\|_2 \le \|Px - Py\|_2 \le (1+\e) \|x-y\|_2 \quad \text{for all } x,y \in \XX.
  \end{equation}
\end{theorem}

\begin{figure}[htp]			
  \centering \includegraphics[width=0.55\textwidth]{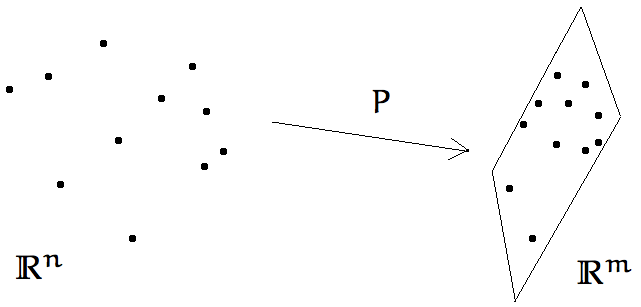} 
  \caption{Johnson-Lindenstrauss Lemma states that a random projection of $N$ data points
    from dimension $n$ to dimension $m \sim \log N$ approximately preserves the distances between the points.}
  \label{fig: JL}	
\end{figure}

\begin{proof}
Take a closer look at the desired conclusion \eqref{eq: JL}. By linearity, $Px-Py=P(x-y)$. 
So, dividing the inequality by $\|x-y\|_2$, we can rewrite \eqref{eq: JL} in the following way: 
\begin{equation}         \label{eq: JL z}
1-\e \le \|Pz\|_2 \le 1+\e \quad \text{for all } z \in T
\end{equation}
where 
$$
T := \left\{ \frac{x-y}{\|x-y\|_2} :\; x,y \in \XX \text{ distinct points} \right\}.
$$
It will be convenient to square the inequality \eqref{eq: JL z}.
Using that $1+\e \le (1+\e)^2$ and $1-\e \ge (1-\e)^2$, 
we see that it is enough to show that 
\begin{equation}         \label{eq: JL squares}
1-\e \le \|Pz\|_2^2 \le 1+\e \quad \text{for all } z \in T.
\end{equation}
By construction, the coordinates of the vector $Pz = \frac{1}{\sqrt{m}} Az$ are 
$\frac{1}{\sqrt{m}} \ip{X_i}{z}$. Thus we can restate \eqref{eq: JL squares} as
\begin{equation}         \label{eq: JL sum}
\Big| \frac{1}{m} \sum_{i=1}^m \ip{X_i}{z}^2 - 1 \Big| \le \e \quad \text{for all } z \in T.
\end{equation}

Results like \eqref{eq: JL sum} are often proved by combining {\em concentration} and a {\em union bound}.
In order to use concentration, we first fix $z \in T$. By assumption, the random variables $\ip{X_i}{z}^2 - 1$ are independent; they have 
zero mean (use isotropy to check this!), and they are sub-exponential (use \eqref{eq: psione psitwo} to check this).
Then Bernstein's inequality (Theorem~\ref{thm: Bernstein}) gives 
$$
\Pr{ \Big| \frac{1}{m} \sum_{i=1}^m \ip{X_i}{z}^2 - 1 \Big| > \e } 
\le 2 \exp (-c \e^2 m). 
$$
(Check!)

Finally, we can unfix $z$ by taking a union bound over all possible $z \in T$: 
\begin{align}
\Pr{ \max_{z \in T} \Big| \frac{1}{m} \sum_{i=1}^m \ip{X_i}{z}^2 - 1 \Big| > \e } 
  &\le \sum_{z \in T} \Pr{ \Big| \frac{1}{m} \sum_{i=1}^m \ip{X_i}{z}^2 - 1 \Big| > \e } \nonumber\\
  &\le |T| \cdot 2 \exp (-c \e^2 m).   \label{eq: JL union bound}
\end{align}
By definition of $T$, we have $|T| \le N^2$. So, if we choose $m \ge C \e^{-2} \log N$
with appropriately large constant $C$, we can make \eqref{eq: JL union bound}
bounded by $0.01$. The proof is complete.
\end{proof}

\subsection{Notes}
%-------------

The material presented in Sections~\ref{s: sub-gaussian}--\ref{s: sub-gaussian vectors} is basic and 
can be found e.g. in \cite{V-RMT-tutorial} and \cite{V-textbook} with all the proofs. 
Bernstein's and Hoeffding's inequalities that we covered here are two basic examples of {\em concentration inequalities}. 
There are many other useful concentration inequalities for sums of independent random variables 
(e.g. Chernoff's and Bennett's) and for more general objects. The textbook \cite{V-textbook} is an elementary 
introduction into concentration; the books \cite{BLM, Ledoux-Talagrand, Ledoux} offer more 
comprehensive and more advanced accounts of this area.

The original version of Johnson-Lindenstrauss Lemma was proved in \cite{JL}. 
The version we gave here, Theorem~\ref{thm: JL}, was stated with probability of success $0.99$,
but an inspection of the proof gives probability $1-2\exp(-c\e^2 m)$ which is much better for large $m$.
A great variety of ramifications and applications of Johnson-Lindenstrauss lemma are known, 
see e.g. \cite{Matousek, Achlioptas, Ailon-Chazelle, BDDW, BLM, Kane-Nelson}.

%%%%%%%%%%%%%%%%%%%
\section{Lecture 2: Concentration of sums of independent random matrices}			\label{s: Lecture 2}
%%%%%%%%%%%%%%%%%%%

In the previous lecture we proved Bernstein's inequality, which quantifies how 
a sum of independent random variables concentrates about its mean. 
We will now study an extension of Bernstein's inequality to higher dimensions, 
which holds for sums of independent {\em random matrices}.

\subsection{Matrix calculus}			\label{s: matrix calculus}
%-------

The key idea of developing a matrix Bernstein's inequality will be to use {\em matrix calculus}, 
which allows us to operate with matrices as with scalars -- adding and multiplying them of course, 
but also comparing matrices and applying functions to matrices. Let us explain this.

We can compare matrices to each other using the notion of being {\em positive semidefinite}. 
Let us focus here on $n \times n$ symmetric matrices.
If $A-B$ is a positive semidefinite matrix,\footnote{Recall that a symmetric real $n \times n$ 
  matrix $M$ is called positive semidefinite if $x^\tran M x \ge 0$ for any vector $x \in \R^n$.}  
which we denote $A - B \succeq 0$, then 
we say that $A \succeq B$ (and, of course, $B \preceq A$). This defines a {\em partial order} 
on the set of $n \times n$ symmetric matrices. The term ``partial'' indicates that, unlike the real numbers, 
there exist $n \times n$ symmetric matrices $A$ and $B$ that can not be compared. 
(Give an example where neither $A \preceq B$ nor $B \succeq A$!) 

Next, let us guess how to measure the {\em magnitude} of a matrix $A$.
The magnitude of a scalar $a \in \R$ is measured by the absolute value $|a|$;  
it is the smallest non-negative number $t$ such that
$$
-t \le a \le t.
$$
Extending this reasoning to matrices, we can measure the magnitude of 
an $n \times n$ symmetric matrix $A$ by the smallest non-negative number $t$ 
such that\footnote{Here and later, $I_n$ denotes the $n \times n$ identity matrix.}
$$
-t I_n \preceq A \preceq t I_n.
$$
The smallest $t$ is called the {\em operator norm} of $A$ and is denoted $\|A\|$.
Diagonalizing $A$, we can see that 
\begin{equation}         \label{eq: norm max eig}
\|A\| = \max \{ |\l| :\; \text{$\l$ is an eigenvalue of $A$} \}.
\end{equation}
With a little more work (do it!), we can see that $\|A\|$ is the norm of $A$
acting as a linear operator on the Euclidean space $(\R^n, \|\cdot\|_2)$;
this is why $\|A\|$ is called the operator norm. Thus $\|A\|$ is the smallest non-negative number $M$ 
such that 
$$
\|Ax\|_2 \le M \|x\|_2 \quad \text{for all } x \in \R^n.
$$

Finally, we will need to be able to take {\em functions of matrices}. 
Let $f : \R \to \R$ be a function and $X$ be an $n \times n$ symmetric matrix. 
We can define $f(X)$ in two equivalent ways. 
The spectral theorem allows us to represent $X$ as 
$$
X = \sum_{i=1}^n \l_i u_i u_i^\tran
$$
where $\l_i$ are the eigenvalues of $X$ and $u_i$ are the corresponding eigenvectors.
Then we can simply define
$$
f(X) := \sum_{i=1}^n f(\l_i) u_i u_i^\tran.
$$
Note that $f(X)$ has the same eigenvectors as $X$, but the eigenvalues change under the action of $f$.
An equivalent way to define $f(X)$ is using power series. Suppose the function $f$ has 
a convergent power series expansion about some point $x_0 \in \R$, i.e.
$$
f(x) = \sum_{k=1}^\infty a_k (x-x_0)^k.
$$
Then one can check that the following matrix series converges\footnote{The convergence holds 
  in any given metric on the set of matrices, for example in the metric given by the operator norm.
  In this series, the terms $(X-X_0)^k$ are defined by the usual matrix product.}
and defines $f(X)$:  
$$
f(X) = \sum_{k=1}^\infty a_k (X-X_0)^k.
$$
(Check!)

\subsection{Matrix Bernstein's inequality} 		\label{s: matrix Bernstein sub}
%---------------

We are now ready to state and prove a remarkable generalization of Bernstein's inequality 
for random matrices. 
 
\begin{theorem}[Matrix Bernstein's inequality]  \label{thm: matrix Bernstein}
  Let $X_1,\ldots,X_N$ be independent, mean zero, $n \times n$ symmetric random matrices, 
  such that $\|X_i\| \le K$ almost surely for all $i$.
  Then, for every $t \ge 0$ we have
  $$
  \P \Big\{ \Big\| \sum_{i=1}^N X_i \Big\| \ge t \Big\} 
  \le 2n \cdot \exp \Big( -\frac{t^2/2}{\s^2 + Kt/3} \Big).
  $$
  Here $\s^2 = \left\| \sum_{i=1}^N \E X_i^2 \right\|$ is the norm of the ``matrix variance'' of the sum.
\end{theorem}

The scalar case, where $n=1$, is the classical Bernstein's inequality we stated in \eqref{eq: Bernstein bounded}.
A remarkable feature of matrix Bernstein's inequality, which makes it especially powerful,
is that it {\em does not require any independence of the entries} (or the rows or columns) of $X_i$; 
all is needed is that the random matrices $X_i$ be independent {\em from each other}.

In the rest of this section we will prove matrix Bernstein's inequality, and give a few applications in this and next lecture.  

Our proof will be based on bounding the {\em moment generating function} (MGF)
$\E \exp(\l S)$ of the sum $S = \sum_{i=1}^N X_i$. Note that to exponentiate the matrix $\l S$ in order to define 
the matrix MGF, we rely on the matrix calculus that we introduced in Section~\ref{s: matrix calculus}.

If the terms $X_i$ were {\em scalars}, independence would yield the classical fact
that the MGF of a product is the product of MGF's, i.e. 
\begin{equation}         \label{eq: MGF sum scalar}
  \E \exp(\l S) = \E \prod_{i=1}^N \exp(\l X_i) = \prod_{i=1}^N \E \exp(\l X_i).
\end{equation}
But for {\em matrices}, this reasoning breaks down badly, for in general
$$
e^{X+Y} \ne e^{X} e^{Y}
$$
even for $2 \times 2$ symmetric matrices $X$ and $Y$. (Give a counterexample!)

Fortunately, there are some trace inequalities that can often serve as proxies for the missing 
equality $e^{X+Y} = e^{X} e^{Y}$.
One of such proxies is the {\em Golden-Thompson inequality}, which states that 
\begin{equation}         \label{eq: Golden-Thompson}
\tr (e^{X+Y}) \le \tr(e^X e^Y)
\end{equation}
for any $n \times n$ symmetric matrices $X$ and $Y$.
Another result, which we will actually use in the proof of matrix Bernstein's inequality, 
is {\em Lieb's inequality}.

\begin{theorem}[Lieb's inequality]	\label{thm: Lieb}
  Let $H$ be an $n \times n$ symmetric matrix. Then the function 
  $$
  f(X) = \tr \exp(H + \log X)
  $$
  is concave\footnote{Formally, {\em concavity} of $f$ means
    that $f(\l X + (1-\l)Y) \ge \l f(X) + (1-\l) f(Y)$ for all symmetric matrices $X$ and $Y$ and all $\l \in [0,1]$.} 
  on the space on $n \times n$ symmetric matrices. 
\end{theorem}

Note that in the scalar case, where $n=1$, the function $f$ in Lieb's inequality is linear and the result is trivial.  

To use Lieb's inequality in a probabilistic context, we will combine it with the classical Jensen's inequality. 
It states that for any concave function $f$ and a random matrix $X$, one has\footnote{Jensen's inequality is 
  usually stated for a {\em convex} function $g$ and a {\em scalar} random variable $X$, and it reads 
  $g(\E X) \le \E g(X)$. From this, inequality \eqref{eq: Jensen} for concave functions and random matrices 
  easily follows (Check!).}
\begin{equation}         \label{eq: Jensen}
\E f(X) \le f(\E X).
\end{equation}
Using this for the function $f$ in Lieb's inequality, we get
$$
\E \tr \exp(H + \log X) \le \tr \exp(H + \log \E X).
$$
And changing variables to $X = e^Z$, we get the following:

\begin{lemma}[Lieb's inequality for random matrices]		\label{lem: Lieb}
  Let $H$ be a fixed $n \times n$ symmetric matrix and $Z$ be an $n \times n$ symmetric random matrix.
  Then 
  $$
  \E \tr \exp (H+Z) \le \tr \exp(H + \log \E e^Z).
  $$
\end{lemma}

Lieb's inequality is a perfect tool for bounding the MGF of a sum of independent random variables
$S = \sum_{i=1}^N X_i$. To do this, let us condition on the random variables $X_1,\ldots,X_{N-1}$.
Apply Lemma~\ref{lem: Lieb} for the fixed matrix $H := \sum_{i=1}^{N-1} \l X_i$ and the random 
matrix $Z := \l X_N$, and afterwards take the expectation with respect to $X_1,\ldots,X_{N-1}$. By the 
law of total expectation, we get
$$
\E \tr \exp(\l S) \le \E \tr \exp \Big[ \sum_{i=1}^{N-1} \l X_i + \log \E e^{\l X_N} \Big].
$$
Next, apply Lemma~\ref{lem: Lieb} in a similar manner for 
$H := \sum_{i=1}^{N-2} \l X_i + \log \E e^{\l X_N}$ and $Z := \l X_{N-1}$, and so on. 
After $N$ times, we obtain:

\begin{lemma}[MGF of a sum of independent random matrices] 	\label{lem: MGF sum matrix}
  Let $X_1,\ldots,X_N$ be independent $n \times n$ symmetric random matrices.
  Then the sum $S = \sum_{i=1}^N X_i$ satisfies
  $$
  \E \tr \exp(\l S) \le \tr \exp \Big[ \sum_{i=1}^N \log \E e^{\l X_i} \Big].
  $$
\end{lemma}

Think of this inequality is a matrix version of the scalar identity \eqref{eq: MGF sum scalar}. 
The main difference is that it bounds the trace of the MGF\footnote{Note that 
  the order of expectation and trace can be swapped using linearity.} 
rather the MGF itself. 

You may recall from a course in probability theory that the quantity $\log \E e^{\l X_i}$ 
that appears in this bound is called the {\em cumulant generating function} of $X_i$.

Lemma~\ref{lem: MGF sum matrix} reduces the complexity of our task significantly,
for it is much easier to bound the cumulant generating function of each {\em single} random variable $X_i$
than to say something about their sum. Here is a simple bound. 

\begin{lemma}[Moment generating function]			\label{lem: matrix MGF}
  Let $X$ be an $n \times n$ symmetric random matrix. 
  Assume that $\E X = 0$ and $\|X\| \le K$ almost surely.
  Then, for all $0 < \l < 3/K$ we have 
  $$
  \E \exp(\l X) \preceq \exp \left( g(\lambda) \E X^2 \right) 
  \quad \text{where} \quad
  g(\lambda) = \frac{\l^2/2}{1-\l K/3}.
  $$ 
\end{lemma}

\begin{proof}
First, check that the following scalar inequality holds for $0 < \l < 3/K$ and $|x| \le K$:
$$
e^{\l x} \le 1 + \l x + g(\l) x^2.
$$
Then extend it to matrices using matrix calculus: if $0 < \l < 3/K$ and $\|X\| \le K$ then 
$$
e^{\l X} \preceq I + \l X + g(\l) X^2.
$$
(Do these two steps carefully!)
Finally, take the expectation and recall that $\E X = 0$ to obtain
$$
\E e^{\l X} \preceq I + g(\l) \E X^2 \preceq \exp \left( g(\lambda) \E X^2 \right).
$$
In the last inequality, we use the matrix version of the scalar inequality $1+z \le e^z$
that holds for all $z \in \R$. 
The lemma is proved.
\end{proof}

\begin{proof}[Proof of Matrix Bernstein's inequality]
We would like to bound the operator norm of the random matrix $S = \sum_{i=1}^N X_i$, 
which, as we know from \eqref{eq: norm max eig}, is the largest eigenvalue of $S$ {\em by magnitude}. 
For simplicity of exposition, let us drop the absolute value from \eqref{eq: norm max eig} and 
just bound the maximal eigenvalue of $S$, which we denote $\l_{\max}(S)$.
(Once this is done, we can repeat the argument for $-S$ to reinstate the absolute value. Do this!)
So, we are to bound
\begin{align*}
\Pr{ \l_{\max}(S) \ge t } 
  &= \Pr{ e^{\l \cdot \l_{\max}(S)} \ge e^{\l t} }  \quad \text{(multiply by $\l>0$ and exponentiate)}\\
  &\le e^{-\l t} \, \E e^{\l \cdot \l_{\max}(S)} \quad \text{(by Markov's inequality)}\\
  &= e^{-\l t} \, \E \l_{\max} (e^{\l S}) \quad \text{(check!)}\\
  &\le e^{-\l t} \, \E \tr e^{\l S} \quad \text{(max of eigenvalues is bounded by the sum)}\\
  &\le e^{-\l t} \tr \exp \Big[ \sum_{i=1}^N \log \E e^{\l X_i} \Big] 
      \quad \text{(use Lemma~\ref{lem: MGF sum matrix})}\\
  &\le \tr \exp \left[ -\l t + g(\l) Z \right] \quad \text{(by Lemma~\ref{lem: matrix MGF})}
\end{align*}
where 
$$
Z := \sum_{i=1}^{N} \E X_i^2.
$$

It remains to optimize this bound in $\l$. The minimum is attained for $\l = t/(\s^2 + Kt/3)$. (Check!)
With this value of $\l$, we conclude 
$$
\Pr{ \l_{\max}(S) \ge t } 
\le n \cdot \exp \Big( -\frac{t^2/2}{\s^2 + Kt/3} \Big).
$$
This completes the proof of Theorem~\ref{thm: matrix Bernstein}. 
\end{proof}

Bernstein's inequality gives a powerful {\em tail bound} for $\| \sum_{i=1}^N X_i \|$. 
This easily implies a useful bound on the {\em expectation}:

\begin{corollary}[Expected norm of sum of random matrices]  \label{cor: matrix Bernstein expectation}
  Let~$X_1,\ldots,X_N$ be independent, mean zero, $n \times n$ symmetric random matrices, 
  such that $\|X_i\| \le K$ almost surely for all $i$.
  Then 
  $$
  \E \Big\| \sum_{i=1}^N X_i \Big\| 
  \lesssim \s \sqrt{\log n} + K \log n
  $$
  where $\s = \big\| \sum_{i=1}^N \E X_i^2 \big\|^{1/2}$.
\end{corollary}

\begin{proof}
The link from tail bounds to expectation is provided by the basic identity 
\begin{equation}         \label{eq: integral}
\E Z = \int_0^\infty \Pr{Z > t} \, dt
\end{equation}
which is valid for any non-negative random variable $Z$. (Check it!)
Integrating the tail bound given by matrix Bernstein's inequality, you 
will arrive at the expectation bound we claimed. (Check!)
\end{proof}

Notice that the bound in this corollary has mild, logarithmic, dependence on the ambient dimension $n$. 
As we will see shortly, this can be an important feature in some applications.

\subsection{Community recovery in networks}		\label{s: networks}
%-----------

Matrix Bernstein's inequality has many applications. The one we are going to discuss first
is for the analysis of {\em networks}. A network can be mathematically represented by a graph, 
a set of $n$ vertices with edges connecting some of them. 
For simplicity, we will consider undirected graphs where the edges do not have arrows.
Real world networks often tend to have clusters, or {\em communities} --
subsets of vertices that are connected by unusually many edges. (Think, for example, about 
a friendship network where communities form around some common interests.) 
An important problem in data science is to recover communities from a given network. 

We are going to explain one of the simplest methods for community recovery, which is 
called {\em spectral clustering}. But before we introduce it, we will first of all place a 
probabilistic model on the networks we consider. In other words, it will be convenient for us 
to view networks as {\em random graphs} whose edges are formed at random. Although not all 
real-world networks are truly random, this simplistic model can motivate us to develop algorithms
that may empirically succeed also for real-world networks.

The basic probabilistic model of random graphs is the {\em Erd\"os-R\'enyi model}. 

\begin{definition}[Erd\"os-R\'enyi model]
  Consider a set of $n$ vertices and connect every pair of vertices independently and with fixed probability $p$. 
  The resulting random graph is said to follow the {\em Erd\"os-R\'enyi model} $G(n,p)$.
\end{definition}

The Erd\"os-R\'enyi random model is very simple. But it is not a good choice if we want to model 
a network with communities, for every pair of vertices has the same chance to be connected. 
So let us introduce a natural generalization of the Erd\"os-R\'enyi random model that does allow 
for community structure:

\begin{definition}[Stochastic block model]		\label{def: SBM}
  Partition a set of $n$ vertices into two subsets (``communities'') 
  with $n/2$ vertices each, and connect every pair of vertices independently with probability $p$ 
  if they belong to the same community and $q<p$ if not. The resulting random graph is said
  to follow the {\em stochastic block model} $G(n,p,q)$.
\end{definition}

Figure~\ref{fig: SBM} illustrates a simulation of a stochastic block model.
\begin{figure}[htp]			
  \centering \includegraphics[width=0.45\textwidth]{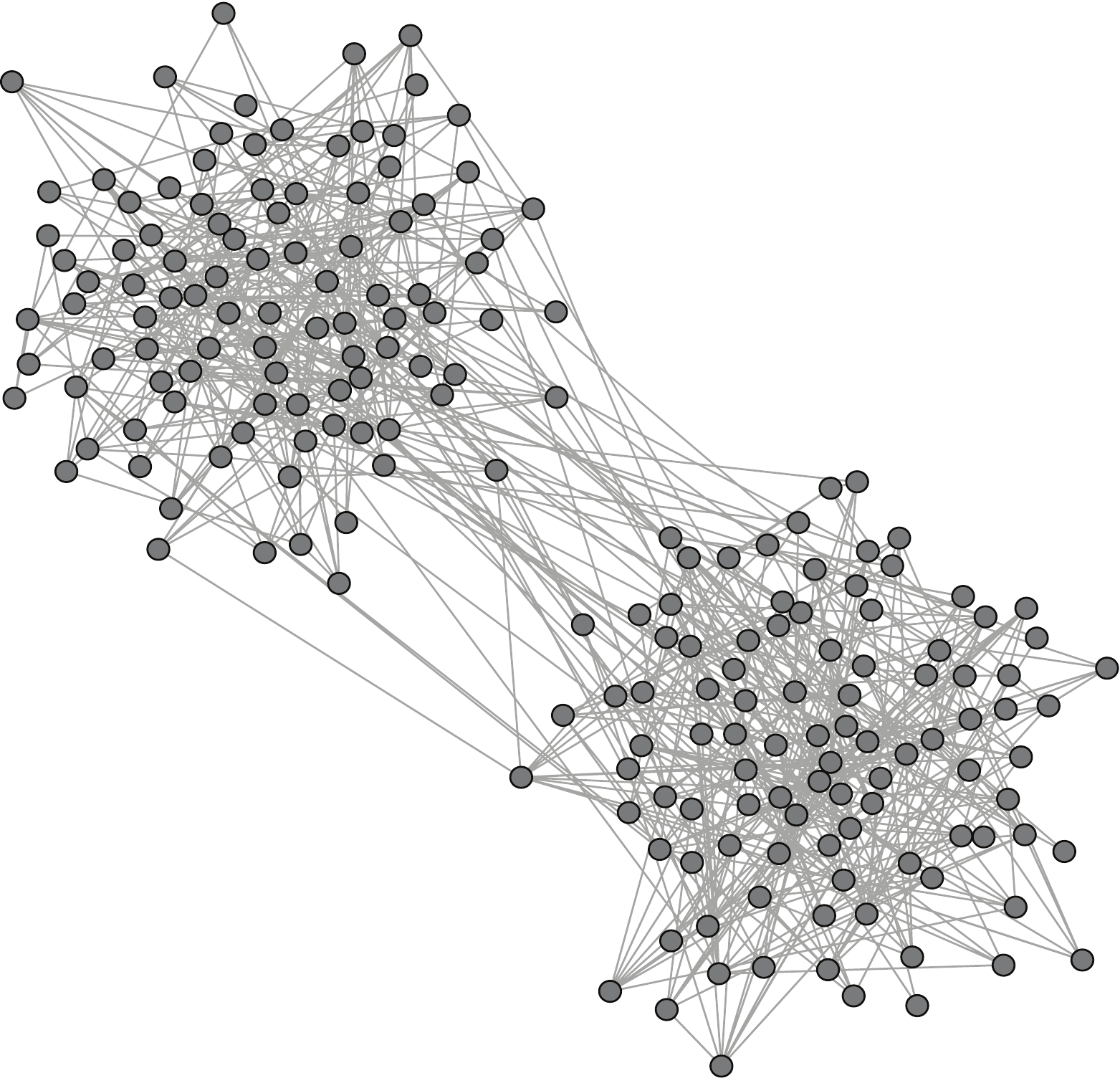} 
  \caption{A network generated according to the stochastic block model $G(n,p,q)$ 
    with $n = 200$ nodes and connection probabilities $p = 1/20$ and $q = 1/200$.}
  \label{fig: SBM}	
\end{figure}

Suppose we are shown one instance of a random graph generated according to 
a stochastic block model $G(n,p,q)$. How can we find which vertices belong to which community?

The {\em spectral clustering} algorithm we are going to explain will do precisely this. 
It will be based on the spectrum of the {\em adjacency matrix} $A$ of the graph, which is the 
$n \times n$ symmetric matrix whose entries $A_{ij}$ equal $1$ if the vertices $i$ and $j$ are connected by an edge,
and $0$ otherwise.\footnote{For convenience, we call the vertices of the graph $1,2,\ldots,n$.}

The adjacency matrix $A$ is a random matrix. Let us compute its expectation first. 
This is easy, since the entires of $A$ are Bernoulli random variables. If $i$ and $j$ belong to the same community then $\E A_{ij} = p$ and otherwise $\E A_{ij} = q$. Thus $A$ has block structure: for example, 
if $n=4$ then $A$ looks like this:
$$
\E A 
= \left[\begin{array}{cc|cc}
  p & p & q & q \\  
  p & p & q & q \\ \hline 
  q & q & p & p \\
  q & q & p & p 
\end{array}\right]
$$
(For illustration purposes, we grouped the vertices from each community together.)

You will easily check that $A$ has rank $2$, and the non-zero eigenvalues
and the corresponding eigenvectors are
\begin{equation}         \label{eq: eigs D}
  \l_1(\E A) = \Big( \frac{p+q}{2} \Big) n , \quad 
  v_1(\E A) = 
  \left[\begin{array}{c}
    1 \\ 1 \\ \hline 1 \\ 1
  \end{array}\right]; \quad
  \l_2(\E A) = \Big( \frac{p-q}{2} \Big) n, \quad 
  v_2(\E A) = 
  \left[\begin{array}{r}
    1 \\ 1 \\ \hline -1 \\ -1
  \end{array}\right].
\end{equation}
(Check!)

The eigenvalues and eigenvectors of $\E A$ tell us a lot about the community structure of the underlying graph. 
Indeed, the first (larger) eigenvalue,
$$
d :=  \Big( \frac{p+q}{2} \Big) n,
$$
is the {\em expected degree} of any vertex of the graph.\footnote{The degree of the vertex 
  is the number of edges connected to it.}
The second eigenvalue tells us whether there is any community structure at all 
(which happens when $p \ne q$ and thus $\l_2(\E A) \ne 0$).
The first eigenvector $v_1$ is not informative of the structure of the network at all. 
It is the second eigenvector $v_2$ that tells us exactly how to separate the vertices into the two communities:
the signs of the coefficients of $v_2$ can be used for this purpose.

Thus if we know $\E A$, we can recover the community structure of the network from the 
signs of the second eigenvector.
The problem is that we do not know $\E A$. Instead, we know the adjacency matrix $A$. If, by some chance, 
$A$ is not far from $\E A$, we may hope to use the $A$ to approximately recover the community structure. 
So is it true that $A \approx \E A$?
The answer is yes, and we can prove it using matrix Bernstein's inequality.

\begin{theorem}[Concentration of the stochastic block model]		\label{thm: conc sbm}
  Let $A$ be the adjacency matrix of a $G(n,p,q)$ random graph. Then 
  $$
  \E \|A-\E A\| \lesssim \sqrt{d \log n} + \log n. 
  $$
  Here $d= (p+q)n/2$ is the expected degree.
\end{theorem}

\begin{proof}
Let us sketch the argument. To use matrix Bernstein's inequality, 
let us break $A$ into a sum of independent random matrices
$$
A = \sum_{i,j: \; i \le j} X_{ij},
$$
where each matrix $X_{ij}$ contains a pair of symmetric entries of $A$, or one diagonal 
entry.\footnote{Precisely, if $i \ne j$, then $X_{ij}$ has all zero entries 
  except the $(i,j)$ and $(j,i)$ entries that can potentially equal $1$. If $i=j$, the only non-zero entry of $X_{ij}$ 
  is the $(i,i)$.}
Matrix Bernstein's inequality obviously applies for the sum 
$$
A-\E A = \sum_{i \le j} (X_{ij} - \E X_{ij}).
$$
Corollary~\ref{cor: matrix Bernstein expectation} gives\footnote{We will liberally use the notation $\lesssim$ 
  to hide constant factors appearing in the inequalities. Thus, $a \lesssim b$ means that 
  $a \le Cb$ for some constant $C$.}
\begin{equation}         \label{eq: conc almost}
\E \|A-\E A\| \lesssim \s \sqrt{\log n} + K \log n
\end{equation}
where $\s^2 = \big\| \sum_{i \le j} \E (X_{ij}-\E X_{ij})^2 \big\|$ and $K = \max_{ij} \|X_{ij} - \E X_{ij}\|$.
It is a good exercise to check that 
$$
\s^2 \lesssim d \quad \text{and} \quad K \le 2.
$$
(Do it!) 
Substituting into \eqref{eq: conc almost}, we complete the proof.
\end{proof}

How useful is Theorem~\ref{thm: conc sbm} for community recovery?
Suppose that the network is not too sparse, namely
$$
d \gg \log n.
$$
Then 
$$
\|A - \E A\| \lesssim \sqrt{d \log n} \quad \text{while} \quad \|\E A\| = \l_1(\E A) = d,
$$
which implies that 
$$
\|A - \E A\| \ll \|\E A\|.
$$
In other words, $A$ nicely approximates $\E A$: the relative error or approximation 
is small in the operator norm. 

At this point one can apply classical results from the {\em perturbation theory} for matrices, 
which state that since $A$ and $\E A$ are close, their eigenvalues and eigenvectors 
must also be close. The relevant perturbation results are {\em Weyl's inequality} for eigenvalues  
and {\em Davis-Kahan's inequality} for eigenvectors, which we will not reproduce here.
Heuristically, what they give us is 
$$
v_2(A) \approx v_2(\E A) =   
\left[\begin{array}{r}
    1 \\ 1 \\ \hline -1 \\ -1
  \end{array}\right].
$$
Then we should expect that most of the coefficients of $v_2(A)$ are positive 
on one community and negative on the other. 
So we can use $v_2(A)$ to approximately recover the communities. This method is called 
{\em spectral clustering}:

\medskip

{\bf Spectral Clustering Algorithm.}
{\em Compute $v_2(A)$, the eigenvector corresponding to the second largest 
  eigenvalue of the adjacency matrix $A$ of the network. 
  Use the signs of the coefficients of $v_2(A)$ to predict the community membership
  of the vertices.}
  
\medskip

We saw that spectral clustering should perform well for the stochastic block model 
$G(n,p,q)$ if it is not too sparse, namely if the expected degrees satisfy
$d = (p+q)n/2 \gg \log n$.

A more careful analysis along these lines, which you should be able to do yourself with some work,
leads to the following more rigorous result. 

\begin{theorem}[Guarantees of spectral clustering]		\label{thm: spectral clustering}
  Consider a random graph generated according to the stochastic block model $G(n,p,q)$
  with $p>q$, and set $a = pn$, $b = qn$. Suppose that 
  \begin{equation}         \label{eq: spectral clustering}
  (a-b)^2 \gg \log(n) (a+b).
  \end{equation}
  Then, with high probability, the spectral clustering algorithm recovers the 
  communities up to $o(n)$ misclassified vertices. 
\end{theorem}

Note that condition \eqref{eq: spectral clustering} implies that the expected degrees are not too small, 
namely $d = (a+b)/2 \gg \log(n)$ (check!). It also ensures that $a$ and $b$ are sufficiently different:
recall that if $a=b$ the network is Erd\"os-R\'enyi graph without any community structure. 

\subsection{Notes}
%-----------

The idea to extend concentration inequalities like Bernstein's to matrices 
goes back to R.~Ahlswede and A.~Winter \cite{Ahlswede-Winter}. 
They used Golden-Thompson inequality \eqref{eq: Golden-Thompson} and proved a slightly weaker form 
of matrix Bernstein's inequality than we gave in Section~\ref{s: matrix Bernstein sub}.
R.~Oliveira \cite{Oliveira1, Oliveira2} found a way to improve this argument and gave a result 
similar to Theorem~\ref{thm: matrix Bernstein}. 
The version of matrix Bernstein's inequality we gave here (Theorem~\ref{thm: matrix Bernstein}) 
and a proof based on Lieb's inequality is due to J.~Tropp \cite{Tropp}. 

The survey \cite{Tropp-book} contains a comprehensive introduction of matrix calculus, 
a proof of Lieb's inequality (Theorem~\ref{thm: Lieb}), a detailed proof of matrix Bernstein's inequality
(Theorem~\ref{thm: matrix Bernstein}) and a variety of applications. 
A proof of Golden-Thompson inequality \eqref{eq: Golden-Thompson} can be found in 
\cite[Theorem~9.3.7]{Bhatia}. 

In Section~\ref{s: networks} we scratched the surface of an interdisciplinary area of {\em network analysis}. 
For a systematic introduction into networks, refer to the book \cite{Newman}.
Stochastic block models (Definition~\ref{def: SBM}) were introduced in \cite{HLL}.
The community recovery problem in stochastic block models, sometimes also called community detection problem, 
has been in the spotlight in the last few years. A vast and still growing body of literature exists on algorithms
and theoretical results for community recovery, see the book \cite{Newman}, the survey \cite{Fortunato-Hric}, 
papers such as \cite{Zhou-Zhang, Mossel-Neeman-Sly, Hajek-Wu-Xu, Bordenave-LM, LLV, Guedon-V, JMR}
and the references therein.

A concentration result similar to Theorem~\ref{thm: conc sbm} can be found in \cite{Oliveira1};
the argument there is also based on matrix concentration. This theorem is not quite optimal.
For dense networks, where the expected degree $d$ satisfies $d \gtrsim \log n$, the concentration inequality 
in Theorem~\ref{thm: conc sbm} can be improved to 
\begin{equation}         \label{eq: concentration improved}
\E \|A-\E A\| \lesssim \sqrt{d}.
\end{equation}
This improved bound goes back to the original paper \cite{Feige-Ofek} which studies the simpler Erd\"os-R\'enyi model 
but the results extend to stochastic block models \cite{Chin-Rao-Vu}; 
it can also be deduced from \cite{Bandeira-van-Handel, Hajek-Wu-Xu, LLV}.

If the network is relatively dense, i.e.  $d \gtrsim \log n$, 
one can improve the guarantee \eqref{eq: spectral clustering} of spectral clustering 
in Theorem~\ref{thm: spectral clustering} to 
$$
(a-b)^2 \gg (a+b).
$$
All one has to do is use the improved concentration inequality \eqref{eq: concentration improved}
instead of Theorem~\ref{thm: conc sbm}. Furthermore, in this case there exist algorithms that can 
recover the communities {\em exactly}, i.e. without any misclassified vertices, and with high probability, 
see e.g. \cite{McSherry, Hajek-Wu-Xu, Abbe-Bandeira-Hall, Chin-Rao-Vu}.

For sparser networks, where $d \ll \log n$ and possibly even $d = O(1)$, relativelyfew algorithms 
were known until recently, but now there exist many approaches that provably recover communities
in sparse stochastic block models, see e.g. \cite{Mossel-Neeman-Sly, Bordenave-LM, Chin-Rao-Vu, 
LLV, Guedon-V, JMR, Zhou-Zhang}.

%%%%%%%%%%%%%%%%%%%
\section{Lecture 3: Covariance estimation and matrix completion}	
%%%%%%%%%%%%%%%%%%%

In the last lecture, we proved matrix Bernstein's inequality and gave an application 
for network analysis. We will spend this lecture discussing a couple of other interesting applications
of matrix Bernstein's inequality. In Section~\ref{s: covariance general} we will work on 
{\em covariance estimation}, a basic problem in high-dimensional statistics. 
In Section~\ref{s: norms matrices}, we will derive a useful bound on norms of random matrices, 
which unlike Bernstein's inequality does not require any boundedness assumptions on the distribution. 
We will apply this bound in Section~\ref{s: matrix completion}
for a problem of {\em matrix completion}, where we are shown a small sample of the entries 
of a matrix and asked to guess the missing entries.

\subsection{Covariance estimation}				\label{s: covariance general}
%----------------------
 
Covariance estimation is a problem of fundamental importance in high-dimensional 
statistics. Suppose we have a sample of data points $X_1,\ldots, X_N$ in $\R^n$. 
It is often reasonable to assume that these points are independently 
sampled from the same probability distribution (or ``population'') which is unknown. 
We would like to learn something useful about this distribution. 

Denote by $X$ a random vector that has this (unknown) distribution. 
The most basic parameter of the distribution is the {\em mean} $\E X$. 
One can estimate $\E X$ from the sample by computing 
the {\em sample mean} $\frac{1}{N} \sum_{i=1}^N X_i$. The law of large numbers guarantees
that the estimate becomes tight as the sample size $N$ grows to infinity, i.e. 
$$
\frac{1}{N} \sum_{i=1}^N X_i \to \E X \quad \text{as } N \to \infty.
$$

The next most basic parameter of the distribution is the {\em covariance matrix} 
$$
\Sigma := \E (X-\E X)(X-\E X)^\tran.
$$
This is a higher-dimensional version of the usual notion 
of {\em variance} of a random variable $Z$, which is 
$$
\Var(Z) = \E (Z-\E Z)^2.
$$
The eigenvectors of the covariance matrix of $\Sigma$ are called the {\em principal components}.
Principal components that correspond to large eigenvalues of $\Sigma$ 
are the directions in which the distribution of $X$ is most extended, see Figure~\ref{fig: PCA}. 
These are often the most interesting 
directions in the data. Practitioners often visualize the high-dimensional data by projecting it 
onto the span of a few (maybe two or three) of such principal components; the projection may 
reveal some hidden structure of the data. This method is called {\em Principal Component Analysis} (PCA).

\begin{figure}[htp]			
  \centering \includegraphics[width=0.45\textwidth]{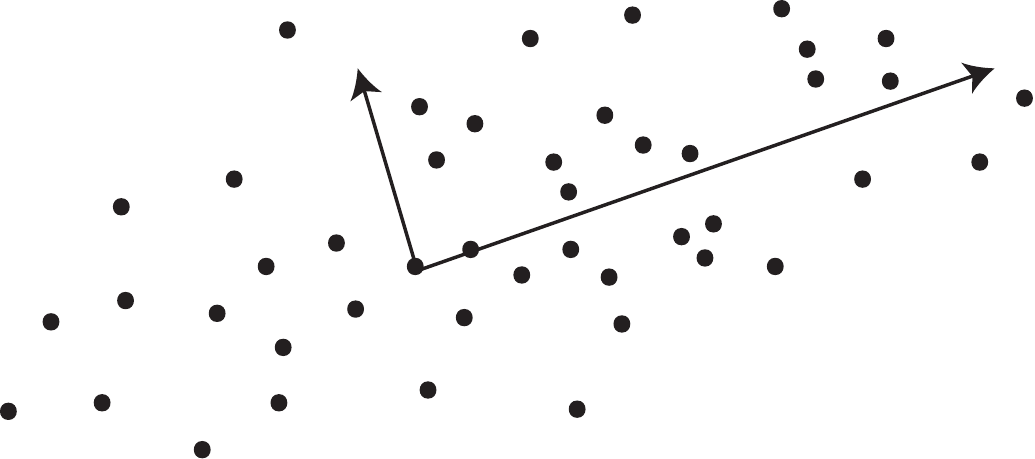} 
  \caption{Data points $X_1,\ldots, X_N$ sampled from a distribution in $\R^n$ and the principal 
    components of the covariance matrix.}
  \label{fig: PCA}	
\end{figure}

One can estimate the covariance matrix $\Sigma$ from the sample by computing 
the {\em sample covariance} 
$$
\Sigma_N := \frac{1}{N} \sum_{i=1}^N (X_i-\E X_i)(X_i - \E X_i)^\tran.
$$
Again, the law of large numbers guarantees
that the estimate becomes tight as the sample size $N$ grows to infinity, i.e. 
$$
\Sigma_N \to \Sigma \quad \text{as } N \to \infty.
$$

But how large should the sample size $N$ be for covariance estimation? 
Generally, one can not have $N < n$ for dimension reasons. (Why?)
We are going to show that 
$$
N \sim n \log  n
$$
is enough. In other words, covariance estimation is possible with just {\em logarithmic oversampling}.

For simplicity, we shall state the covariance estimation bound for {\em mean zero} distributions. 
(If the mean is not zero, we can estimate it from the sample and subtract. 
Check that the mean can be accurately estimated from a sample of size $N = O(n)$.)

\begin{theorem}[Covariance estimation]			\label{thm: covariance estimation general}
  Let $X$ be a random vector in $\R^n$ with covariance matrix $\Sigma$.
  Suppose that 
  \begin{equation}         \label{eq: distribution bounded}
  \|X\|_2^2 \lesssim \E \|X\|_2^2 = \tr \Sigma \quad \text{almost surely}.
  \end{equation}
  Then, for every $N \ge 1$, we have
  $$
  \E \|\Sigma_N - \Sigma\| \lesssim \|\Sigma\| \; \Big( \sqrt{\frac{n \log n}{N}} + \frac{n \log n}{N} \Big).
  $$
\end{theorem}

Before we pass to the proof, let us note that Theorem~\ref{thm: covariance estimation general} 
yields the covariance estimation result we promised. 
Let $\e \in (0,1)$. If we take a sample of size 
$$
N \sim \e^{-2} n \log n,
$$
then we are guaranteed covariance estimation with a good relative error:
$$
\E \|\Sigma_N - \Sigma\| \le \e \|\Sigma\|.
$$

\begin{proof}
Apply matrix Bernstein's inequality (Corollary~\ref{cor: matrix Bernstein expectation}) 
for the sum of independent random matrices $X_i X_i^\tran - \Sigma$ 
and get 
\begin{equation}         \label{eq: covariance prelim}
\E \|\Sigma_N - \Sigma\| 
= \frac{1}{N} \E \Big\| \sum_{i=1}^N (X_i X_i^\tran - \Sigma) \Big\|
\lesssim \frac{1}{N} \big( \s \sqrt{\log n} + K \log n \big)
\end{equation}
where 
$$
\s^2 = \Big\| \sum_{i=1}^N \E(X_i X_i^\tran - \Sigma)^2 \Big\|
= N \big\| \E(XX^\tran - \Sigma)^2 \big\|
$$
and $K$ is chosen so that 
$$
\|XX^\tran - \Sigma\| \le K \quad \text{almost surely}.
$$

It remains to bound $\s$ and $K$. Let us start with $\sigma$. We have 
\begin{align*}
\E(XX^\tran- \Sigma)^2 
  &= \E \|X\|_2^2 XX^\tran - \Sigma^2 \quad \text{(check by expanding the square)} \\
  &\precsim \tr(\Sigma) \cdot \E XX^\tran 
  	\quad \text{(drop $\Sigma^2$ and use \eqref{eq: distribution bounded})} \\
  &= \tr(\Sigma) \cdot \Sigma.
\end{align*}
Thus 
$$
\s^2 \lesssim N \tr(\Sigma) \|\Sigma\|.
$$

Next, to bound $K$, we have
\begin{align*}
\|XX^\tran - \Sigma\| 
  &\le \|X\|_2^2 + \|\Sigma\| \quad \text{(by triangle inequality)} \\
  &\lesssim \tr \Sigma + \|\Sigma\| \quad \text{(using \eqref{eq: distribution bounded})} \\
  &\le 2 \tr \Sigma =: K.
\end{align*}

Substitute the bounds on $\sigma$ and $K$ into \eqref{eq: covariance prelim} 
and get
$$
\E \|\Sigma_N - \Sigma\|
\lesssim \frac{1}{N} \big( \sqrt{N \tr(\Sigma) \|\Sigma\| \log n} + \tr(\Sigma) \log n \big)
$$
To complete the proof, use that $\tr \Sigma \le n \|\Sigma\|$ (check this!) 
and simplify the bound. 
\end{proof}

\begin{remark}[Low-dimensional distributions]
  Much fewer samples are needed for covariance estimation for low-dimensional, 
  or approximately low-dimensional, distributions. To measure approximate low-dimensionality 
  we can use the notion of the {\em stable rank} of $\Sigma^2$. The stable rank of 
  a matrix $A$ is defined as the square of the ratio of the Frobenius to operator 
  norms:\footnote{The Frobenius norm of an $n \times m$ matrix, sometimes also called the {\em Hilbert-Schmidt norm}, 
    is defined as $\|A\|_F = (\sum_{i=1}^n \sum_{j=1}^m A_{ij}^2)^{1/2}$. Equivalently, for an $n \times n$
    symmetric matrix, $\|A\|_F = (\sum_{i=1}^n \l_i(A)^2)^{1/2}$, where $\l_i(A)$ are the eigenvalues of $A$. 
    Thus the stable rank of $A$ can be expressed as $r(A) = \sum_{i=1}^n \l_i(A)^2 / \max_i \l_i(A)^2$.} 
  $$
  r(A) := \frac{\|A\|_F^2}{\|A\|^2}.
  $$
  The stable rank is always bounded by the usual, linear algebraic rank,
  $$
  r(A) \le \rank(A),
  $$
  and it can be much smaller. (Check both claims.) 
  
  Our proof of Theorem~\ref{thm: covariance estimation general} actually gives
  $$
  \E \|\Sigma_N - \Sigma\| \le \|\Sigma\| \; \Big( \sqrt{\frac{r \log n}{N}} + \frac{r \log n}{N} \Big).
  $$
  where 
  $$
  r = r(\Sigma^{1/2}) = \frac{\tr \Sigma}{\|\Sigma\|}.
  $$
  (Check this!) Therefore, covariance estimation is possible with 
  $$
  N \sim r \log n
  $$
  samples. 
\end{remark}

\begin{remark}[The boundedness condition]
  It is a good exercise to check that if we remove the boundedness condition \eqref{eq: distribution bounded}, 
  a nontrivial covariance estimation is impossible in general. (Show this!)
  But how do we know whether the boundedness condition holds for data at hand? We may not, but we can enforce
  this condition by {\em truncation}. All we have to do is to discard $1\%$ of data points
  with largest norms. (Check this accurately, assuming that such truncation does not change 
  the covariance significantly.)
\end{remark}

%Furthermore, we can have even fewer samples if the distribution is low-dimensional. 
%Suppose that the entire distribution is supported by some $r$-dimensional subspace, 
%which is equivalent to the rank constraint 
%$$
%\rank(\Sigma) \le r.
%$$
%We will show that 
%$$
%N \sim r \log n
%$$
%samples are enough for covariance estimation. The same is true if the distribution is 
%{\em approximately low-dimensional}. we can quantify using the {\em stable rank} 
%$$
%r(\Sigma)
%$$

\subsection{Norms of random matrices}			\label{s: norms matrices}
%--------------------------

We have worked a lot with the {\em operator norm} of matrices, denoted $\|A\|$.
One may ask if is there exists a formula that expresses $\|A\|$ in terms of the entires $A_{ij}$.
Unfortunately, there is no such formula. The operator norm is a more difficult quantity in this respect
than the {\em Frobenius norm}, which as we know can be easily expressed in terms of entries:
$\|A\|_F = (\sum_{i,j} A_{ij}^2)^{1/2}$.

If we can not express $\|A\|$ in terms of the entires, can we at least get a good estimate?
Let us consider $n \times n$ symmetric matrices for simplicity. In one direction,
$\|A\|$ is always bounded {\em below} by the largest Euclidean norm 
of the rows $A_i$: 
\begin{equation}         \label{eq: norm lower}
\|A\| \ge \max_i \|A_i\|_2 = \max_i \Big( \sum_j A_{ij}^2 \Big)^{1/2}.
\end{equation}
(Check!)
Unfortunately, this bound is sometimes very loose, and the best possible upper bound is 
\begin{equation}         \label{eq: norm upper}
\|A\| \le \sqrt{n} \cdot  \max_i \|A_i\|_2.
\end{equation}
(Show this bound, and give an example where it is sharp.)

Fortunately, for {\em random} matrices with independent entries 
the bound \eqref{eq: norm upper} can be improved to the point
where the upper and lower bounds almost match.

\begin{theorem}[Norms of random matrices without boundedness assumptions]	 \label{thm: matrix non-bdd}
  Let $A$ be an $n \times n$ symmetric random matrix whose entries
  on and above the diagonal are independent, mean zero random variables. Then 
  $$
  \E \max_i \|A_i\|_2 \le \E \|A\| \le C \log n \cdot \E \max_i \|A_i\|_2,
  $$
  where $A_i$ denote the rows of $A$. 
\end{theorem}

In words, the operator norm of a random matrix is almost determined by the norm of the rows.

Our proof of this result will be based on matrix Bernstein's inequality -- more precisely, 
Corollary~\ref{cor: matrix Bernstein expectation}.
There is one surprising point. How can we use matrix Bernstein's inequality, 
which applies only for bounded distributions, to prove a result like 
Theorem~\ref{thm: matrix non-bdd} that does not have any boundedness assumptions?
We will do this using a trick based on conditioning and symmetrization. Let us introduce this technique first. 

\begin{lemma}[Symmetrization]	\index{Symmetrization}			\label{lem: symmetrization}
  Let $X_1,\ldots,X_N$ be independent, mean zero random vectors in a normed space
  and $\e_1, \ldots, \e_N$ be independent Rademacher random variables.\footnote{This means that 
  random variables $\e_i$ take values $\pm 1$ with probability $1/2$ each. We require that 
  all random variables we consider here, i.e. $\{X_i$, $\e_i: \; i=1,\ldots,N\}$ are jointly independent.}
  Then 
  $$
  \frac{1}{2} \E \Big\| \sum_{i=1}^N \e_i X_i \Big\| 
  \le \E \Big\| \sum_{i=1}^N X_i \Big\|
  \le 2 \E \Big\| \sum_{i=1}^N \e_i X_i \Big\|.
  $$
\end{lemma}

\begin{proof}
To prove the upper bound, let $(X'_i)$ be an independent copy of the random vectors $(X_i)$, 
i.e. just different random vectors with the same joint distribution as $(X_i)$ and independent from $(X_i)$.
Then 
\begin{align*}
\E \Big\| \sum_i X_i \Big\| 
&= \E \Big\| \sum_i X_i - \E \Big( \sum_i X'_i \Big) \Big\| \quad \text{(since $\E \sum_i X'_i = 0$ by assumption)}\\
&\le \E \Big\| \sum_i X_i - \sum_i X'_i \Big\| \quad \text{(by Jensen's inequality)} \\
&= \E \Big\| \sum_i (X_i - X'_i) \Big\|.
\end{align*}

The distribution of the random vectors $Y_i := X_i-X'_i$ is {\em symmetric}, which means that 
the distributions of $Y_i$ and $-Y'_i$ are the same. (Why?)
Thus the distribution of the random vectors $Y_i$ and $\e_i Y_i$ is also the same, for all we do is
change the signs of these vectors at random and independently of the values of the vectors. 
Summarizing, we can replace $X_i - X'_i$ in the sum above with $\e_i(X_i - X'_i)$. Thus
\begin{align*}
\E \Big\| \sum_i X_i \Big\| 
&\le \E \Big\| \sum_i \e_i(X_i - X'_i) \Big\| \\
&\le \E \Big\| \sum_i \e_i X_i \Big\| + \E \Big\| \sum_i \e_i X'_i \Big\| \quad \text{(using triangle inequality)} \\
&= 2 \E \Big\| \sum_i \e_i X_i \Big\| \quad \text{(the two sums have the same distribution)}.
\end{align*}
This proves the upper bound in the symmetrization inequality. 
The lower bound can be proved by a similar argument. (Do this!)
\end{proof}

\begin{proof}[Proof of Theorem~\ref{thm: matrix non-bdd}.]
We already mentioned in \eqref{eq: norm lower} that the bound in Theorem~\ref{thm: matrix non-bdd} is trivial.
The proof of the upper bound will be based on matrix Bernstein's inequality.

First, we decompose $A$ in the same way as we did in the proof of Theorem~\ref{thm: conc sbm}. 
Thus we represent $A$ as a sum of independent, mean zero, symmetric random matrices
$Z_{ij}$ each of which contains a pair of symmetric entries of $A$ (or one diagonal entry): 
$$
A = \sum_{i,j: \; i \le j} Z_{ij}.
$$
Apply the symmetrization inequality (Lemma~\ref{lem: symmetrization}) 
for the random matrices $Z_{ij}$ and get
\begin{equation}         \label{eq: exp A}
\E \|A\| 
= \E \Big\| \sum_{i \le j} Z_{ij} \Big\|
\le 2 \E \Big\| \sum_{i \le j} X_{ij} \Big\|
\end{equation}
where we set
$$
X_{ij} := \e_{ij} Z_{ij}
$$
and $\e_{ij}$ are independent Rademacher random variables. 

Now we condition on $A$. The random variables $Z_{ij}$ become {\em fixed values} 
and all randomness remains in the Rademacher random variables $\e_{ij}$. Note that 
$X_{ij}$ are (conditionally) {\em bounded} almost surely, and this is exactly what 
we have lacked to apply matrix Bernstein's inequality. Now we can do it. 
Corollary~\ref{cor: matrix Bernstein expectation} gives\footnote{We stick a subscript $\e$ 
  to the expected value to remember that this is a conditional expectation, 
  i.e. we average only with respect to $\e_i$.}
\begin{equation}         \label{eq: MBI conditional}
\E_\e \Big\| \sum_{i \le j} X_{ij} \Big\| \lesssim \s \sqrt{\log n} + K \log n,
\end{equation}
where 
$\s^2 = \big\| \sum_{i \le j} \E_\e X_{ij}^2 \big\|$ 
and $K = \max_{i \le j} \|X_{ij}\|$.

A good exercise is to check that  
$$
\s \lesssim \max_i \|A_i\|_2 
\quad \text{and} \quad 
K \lesssim \max_i \|A_i\|_2.
$$
(Do it!)
Substituting into \eqref{eq: MBI conditional}, we get
$$
\E_\e \Big\| \sum_{i \le j} X_{ij} \Big\| \lesssim \log n \cdot \max_i \|A_i\|_2.
$$ 

Finally, we unfix $A$ by taking expectation of both sides of this inequality
with respect to $A$ and using the law of total expectation.
The proof is complete. 
\end{proof}

We stated Theorem~\ref{thm: matrix non-bdd} for symmetric matrices, but it is simple
to extend it to general $m \times n$ random matrices $A$. The bound in this case becomes 
\begin{equation}         \label{eq: matrix non-bdd symmetric}
\E \|A\| \le C \log (m+n) \cdot \big( \E \max_i \|A_i\|_2 + \E \max_j \|A^j\|_2 \big)
\end{equation}
where $A_i$ and $A^j$ denote the rows and columns of $A$.
To see this, apply Theorem~\ref{thm: matrix non-bdd} to the $(m+n) \times (m+n)$ 
symmetric random matrix
$$
\begin{bmatrix}
0 & A \\
A^\tran & 0
\end{bmatrix}.
$$ 
(Do this!)

\subsection{Matrix completion}			\label{s: matrix completion}
%--------------

Consider a fixed, unknown $n \times n$ matrix $X$. Suppose we are shown 
$m$ randomly chosen entries of $X$. Can we guess all the missing entries? 
This important problem is called {\em matrix completion}. We 
will analyze it using the bounds on the norms on random matrices 
we just obtained.

Obviously, there is no way to guess the missing entries unless we know something
extra about the matrix $X$. So let us assume that $X$ has {\em low rank}:
$$
\rank(X) =: r \ll n.
$$
The number of degrees of freedom of an $n \times n$ matrix with rank $r$ is $O(rn)$.
(Why?)
So we may hope that
\begin{equation}         \label{eq: m ideal}
m \sim rn
\end{equation}
observed entries of $X$ will be enough to determine $X$ completely. But how?

Here we will analyze what is probably the simplest method for matrix completion.  
Take the matrix $Y$ that consists of the observed entries of $X$ while all 
unobserved entries are set to zero. 
Unlike $X$, the matrix $Y$ may not have small rank. Compute the best rank $r$ 
approximation\footnote{The best rank $r$ approximation of an $n \times n$ matrix $A$ is 
  a matrix $B$ of rank $r$ that minimizes the operator norm $\|A-B\|$ 
  or, alternatively, the Frobenius norm $\|A-B\|_F$
  (the minimizer turns out to be the same). One can compute $B$ by truncating 
  the singular value decomposition $A = \sum_{i=1}^n s_i u_i v_i^\tran$ of $A$ as follows: 
  $B = \sum_{i=1}^r s_i u_i v_i^\tran$, where we assume that the singular values $s_i$
  are arranged in non-increasing order.}
of $Y$.
The result, as we will show, will be a good approximation 
to $X$.

But before we show this, let us define sampling of entries more rigorously. 
Assume each entry of $X$ is shown or hidden independently of others with fixed probability $p$. 
Which entries are shown is decided by independent Bernoulli random variables 
$$
\d_{ij} \sim \Ber(p) \quad \text{with} \quad p := \frac{m}{n^2}
$$
which are often called {\em selectors} in this context. The value of $p$ is chosen so that
among $n^2$ entries of $X$, the expected number of selected (known) entries is $m$. 
Define the $n \times n$ matrix $Y$ with entries
$$
Y_{ij} := \d_{ij} X_{ij}.
$$
We can assume that we are shown $Y$, for it is a matrix that contains the observed entries of $X$
while all unobserved entries are replaced with zeros. The following result shows how to 
estimate $X$ based on $Y$.

\begin{theorem}[Matrix completion]			\label{thm: matrix completion}
  Let $\hat{X}$ be a best rank $r$ approximation to $p^{-1} Y$. Then 
  \begin{equation}         \label{eq: matrix completion}
  \E \frac{1}{n} \|\hat{X}-X\|_F \le C \log(n) \sqrt{\frac{rn}{m}} \; \|X\|_\infty,
  \end{equation}
  %as long as $m \ge n \log n$. 
  Here $\|X\|_\infty = \max_{i,j}|X_{ij}|$ denotes the maximum magnitude of the entries of $X$.
\end{theorem}

Before we prove this result, let us understand what this bound says about the quality of matrix completion. 
The recovery error is measured in the Frobenius norm, and the left side 
of \eqref{eq: matrix completion} is
$$
\frac{1}{n} \|\hat{X}-X\|_F = \Big( \frac{1}{n^2} \sum_{i,j=1}^n |\hat{X}_{ij} - X_{ij}|^2 \Big)^{1/2}.
$$
Thus Theorem~\ref{thm: matrix completion} controls {\em the average error per entry} in the mean-squared sense. 
To make the error small, let us assume that we have a sample of size
$$
m \gg r n \log^2 n,
$$
which is slightly larger than the ideal size we discussed in \eqref{eq: m ideal}.
This makes $C \log(n) \sqrt{rn/m} = o(1)$ and forces the recovery error to be bounded by $o(1) \|X\|_\infty$. 
Summarizing, Theorem~\ref{thm: matrix completion} says that {\em the expected average error per 
entry is much smaller than the maximal magnitude of the entries of $X$}. This is true for a sample 
of almost optimal size $m$. The smaller the rank $r$ of the matrix $X$, the fewer entries of $X$ 
we need to see in order to do matrix completion.
  
\begin{proof}[Proof of Theorem~\ref{thm: matrix completion}]
{\bf Step 1: The error in the operator norm.}
Let us first bound the recovery error in the {\em operator} norm.
Decompose the error into two parts using triangle inequality: 
$$
\|\hat{X}-X\| \le \|\hat{X} - p^{-1} Y\| + \|p^{-1} Y - X\|.
$$
Recall that $\hat{X}$ is a best approximation to $p^{-1} Y$. 
Then the first part of the error is smaller than the second part, i.e. 
$\|\hat{X} - p^{-1} Y\| \le \|p^{-1} Y - X\|$, and we have
\begin{equation}         \label{eq: removing Xhat}
\|\hat{X}-X\| \le 2 \|p^{-1} Y - X\| = \frac{2}{p} \|Y-pX\|.
\end{equation}
The entries of the matrix $Y-pX$,
$$
(Y-pX)_{ij} = (\d_{ij}-p) X_{ij},
$$
are independent and mean zero random variables. 
Thus we can apply the bound \eqref{eq: matrix non-bdd symmetric} on the norms 
of random matrices and get 
\begin{equation}         \label{eq: via rows cols}
\E \|Y-pX\| \le C \log n \cdot \big( \E \max_{i \in [n]} \|(Y-pX)_i\|_2 + \E \max_{j \in [n]} \|(Y-pX)^j\|_2 \big).
\end{equation}

All that remains is to bound the norms of the rows and columns of $Y-pX$. 
This is not difficult if we note that they can be expressed as sums of independent random variables:  
$$
\|(Y-pX)_i\|_2^2 = \sum_{j=1}^n (\d_{ij}-p)^2 X_{ij}^2 \le \sum_{j=1}^n (\d_{ij}-p)^2 \cdot \|X\|_\infty^2,
$$
and similarly for columns. Taking expectation and noting that $\E(\d_{ij}-p)^2 = \Var(\d_{ij}) = p(1-p)$, 
we get\footnote{The first bound below that compares the $L^1$ and $L^2$ averages follows from 
H\"older's inequality.} 
\begin{equation}         \label{eq: exp rows}
\E \|(Y-pX)_i\|_2 \le (\E \|(Y-pX)_i\|_2^2)^{1/2} \le \sqrt{p n} \, \|X\|_\infty.
\end{equation}
This is a good bound, but we need something stronger in \eqref{eq: via rows cols}. 
Since the maximum appears inside the expectation, we need a {\em uniform} bound, 
which will say that all rows are bounded simultaneously with high probability. 

Such uniform bounds are usually proved by applying concentration inequalities followed 
by a union bound. Bernstein's inequality \eqref{eq: Bernstein bounded} yields
$$
\Pr{ \sum_{j=1}^n (\d_{ij}-p)^2 > t p n } \le \exp(-ctpn) 
\quad \text{for } t \ge 3.
$$
(Check!)
This probability can be further bounded by $n^{-ct}$ using the assumption that $m = pn^2 \ge n \log n$.
A union bound over $n$ rows leads to 
$$
\Pr{ \max_{i \in [n]} \sum_{j=1}^n (\d_{ij}-p)^2 > t p n } \le n \cdot n^{-ct}
\quad \text{for } t \ge 3.
$$
Integrating this tail, we conclude using \eqref{eq: integral} that 
$$
\E \max_{i \in [n]} \sum_{j=1}^n (\d_{ij}-p)^2 \lesssim pn.
$$
(Check!) And this yields the desired bound on the rows,
$$
\E \max_{i \in [n]} \|(Y-pX)_i\|_2 \lesssim \sqrt{pn},
$$
which is an improvement of \eqref{eq: exp rows} we wanted.
We can do similarly for the columns. Substituting into \eqref{eq: via rows cols}, this gives
$$
\E \|Y-pX\| \lesssim \log(n) \sqrt{pn} \; \|X\|_\infty.
$$ 
Then, by \eqref{eq: removing Xhat}, we get
\begin{equation}         \label{eq: error in operator norm}
\E \|\hat{X}-X\| \lesssim \log(n) \sqrt{\frac{n}{p}} \; \|X\|_\infty.
\end{equation}

{\bf Step 2: Passing to Frobenius norm.}
Now we will need to pass from the operator to Frobenius norm. This is where 
we will use for the first (and only) time the rank of $X$. 
We know that $\rank(X) \le r$ by assumption and $\rank(\hat{X}) \le r$ by construction, so
$\rank(\hat{X}-X) \le 2r$. There is a simple relationship between 
the operator and Frobenius norms:
$$
\|\hat{X}-X\|_F \le \sqrt{2r} \|\hat{X}-X\|.
$$
(Check it!)
Take expectation of both sides and use \eqref{eq: error in operator norm}; we get
$$
\E \|\hat{X}-X\|_F \le \sqrt{2r} \E \|\hat{X}-X\| \lesssim \log(n) \sqrt{\frac{rn}{p}} \; \|X\|_\infty.
$$
Dividing both sides by $n$, we can rewrite this bound as 
$$
\E \frac{1}{n} \|\hat{X}-X\|_F \lesssim \log(n) \sqrt{\frac{rn}{pn^2}} \; \|X\|_\infty.
$$
But $pn^2 = m$ by definition of the sampling probability $p$. 
This yields the desired bound \eqref{eq: matrix completion}.
\end{proof}

\subsection{Notes}
%---------------

Theorem~\ref{thm: covariance estimation general} on covariance estimation is 
a version of \cite[Corollary~5.52]{V-RMT-tutorial}, see also \cite{KL}. 
The logarithmic factor is in general necessary. This theorem is a general-purpose result.
If one knows some additional structural information about the covariance matrix (such as sparsity), 
then fewer samples may be needed, see e.g. \cite{Cai-Zhao-Zhou, Levina-V, Chen-Gittens-Tropp}.

A version of Theorem~\ref{thm: matrix non-bdd} was proved in \cite{RS} in a more technical way.
Although the logarithmic factor in Theorem~\ref{thm: matrix non-bdd} can not be completely removed 
in general, it can be improved. Our argument actually gives
$$
\E \|A\| \le C \sqrt{\log n} \cdot \E \max_i \|A_i\|_2 + C \log n \cdot \E \max_{ij} |A_{ij}|,
$$
Using different methods, one can save an extra $\sqrt{\log n}$ factor and show that
$$
\E \|A\| \le C \E \max_i \|A_i\|_2 + C \sqrt{\log n} \cdot \E \max_{ij} |A_{ij}|
$$
(see \cite{Bandeira-van-Handel}) and 
$$
\E \|A\| \le C \sqrt{\log n \cdot \log \log n} \cdot \E \max_i \|A_i\|_2,
$$
see \cite{van-Handel}. (The results in \cite{Bandeira-van-Handel, van-Handel} are stated for Gaussian random matrices; 
the two bounds above can be deduced by using conditioning and symmetrization.)
The surveys \cite{V-RMT-tutorial, Bandeira-van-Handel} and the textbook \cite{V-textbook} present several other useful 
techniques to bound the operator norm of random matrices. 

The matrix completion problem, which we discussed in Section~\ref{s: matrix completion}, 
has attracted a lot of recent attention. E.~Candes and B.~Recht \cite{Candes-Recht} 
showed that one can often achieve {\em exact} matrix completion, thus computing the precise 
values of all missing values of a matrix, from $m \sim r n \log^2(n)$ randomly sampled entries. 
For exact matrix completion, one needs an extra {\em incoherence} assumption that is not present
in Theorem~\ref{thm: matrix completion}. This assumption basically excludes matrices that are 
simultaneously sparse and low rank (such as a matrix whose all but one entries are zero -- it would 
be extremely hard to complete it, since sampling will likely miss the non-zero entry). 
Many further results on exact matrix completion are known, e.g. \cite{Candes-Tao-completion, Recht, Gross, DPVW}.

Theorem~\ref{thm: matrix completion} with a simple proof is borrowed from \cite{PVY};
see also the tutorial \cite{V-estimation-tutorial}. This result only guarantees approximate matrix completion, 
but it does not have any incoherence assumptions on the matrix.

%%%%%%%%%%%%%%%%%%%
\section{Lecture 4: Matrix deviation inequality}	
%%%%%%%%%%%%%%%%%%%

In this last lecture we will study a new uniform deviation inequality for random matrices.
This result will be a far reaching generalization of the Johnson-Linden\-strauss Lemma 
we proved in Lecture~1. 

Consider the same setup as in Theorem~\ref{thm: JL}, 
where $A$ is an $m \times n$ random matrix whose rows are independent, mean zero, 
isotropic and sub-gaussian random vectors in $\R^n$. 
(If you find it helpful to think in terms of concrete examples, let the entries of $A$ be independent 
$N(0,1)$ random variables.) Like in the Johnson-Lindenstrauss Lemma, we will be looking  
at $A$ as a linear transformation from $\R^n$ to $\R^m$, and we will be interested in 
what $A$ does to points in some set in $\R^n$. This time, however, we will allow for {\em infinite} 
sets $T \subset \R^n$. 

Let us start by analyzing what $A$ does to a single fixed vector $x \in \R^n$. We have
\begin{align*}
\E \|Ax\|_2^2 
&= \E \sum_{j=1}^m \ip{A_j}{x}^2 \quad \text{(where $A_j^\tran$ denote the rows of $A$)}\\ 
&= \sum_{j=1}^m \E \ip{A_j}{x}^2 \quad \text{(by linearity)}\\
&= m \|x\|_2^2 \quad \text{(using isotropy of $A_j$).}
\end{align*}
Further, if we assume that concentration about the mean holds here (and in fact, it does), 
we should expect that 
\begin{equation}         \label{eq: deviation single}
\|Ax\|_2 \approx \sqrt{m} \, \|x\|_2
\end{equation}
with high probability.

Similarly to Johnson-Lindenstrauss Lemma, our next goal is to make \eqref{eq: deviation single} hold
simultaneously over all vectors $x$ in some fixed set $T \subset \R^n$. Precisely, 
we may ask -- how large is the average {\em uniform deviation}:
\begin{equation}         \label{eq: deviation question}
\E \sup_{x \in T} \Big| \|Ax\|_2 - \sqrt{m} \, \|x\| \Big| \,?
\end{equation}
This quantity should clearly depend on some notion of the size of $T$: the larger $T$, the larger 
should the uniform deviation be. So, how can we quantify the size of $T$ for this problem?
In the next section we will do precisely this -- introduce a convenient, geometric measure of the 
sizes of sets in $\R^n$, which is called {\em Gaussian width}.

\subsection{Gaussian width}
%------------

\begin{definition}
  Let $T \subset \R^n$ be a bounded set, and $g$ be a standard normal random vector in $\R^n$,
  i.e. $g \sim N(0,I_n)$. Then the quantities  
  $$
  w(T) := \E \sup_{x \in T} \ip{g}{x} 
  \quad \text{and} \quad
  \gamma(T) := \E \sup_{x \in T} |\ip{g}{x}|
  $$
  are called the {\em Gaussian width} of $T$ and the {\em Gaussian complexity} of $T$, 
  respectively. 
\end{definition}

Gaussian width and Gaussian complexity are closely related. Indeed,\footnote{The set $T-T$ is defined as $\{x-y:\; x,y \in T\}$. 
  More generally, given two sets $A$ and $B$ in the same vector space, 
  the {\em Minkowski sum} of $A$ and $B$ is defined as $A+B = \{a+b:\; a \in A, \, b \in B\}$.}
\begin{equation}         \label{eq: width vs complexity}
w(T) = \frac{1}{2} w(T-T) 
= \frac{1}{2} \E \sup_{x,y \in T} \ip{g}{x-y}
= \frac{1}{2} \E \sup_{x,y \in T} |\ip{g}{x-y}| 
= \frac{1}{2} \gamma(T-T).
\end{equation}
(Check these identities!)

Gaussian width has a natural geometric interpretation. 
Suppose $g$ is a unit vector in $\R^n$. Then a moment's thought reveals that $\sup_{x,y \in T} \ip{g}{x-y}$ is
simply the {\em width of $T$ in the direction of $g$}, i.e. the distance between the two
hyperplanes with normal $g$ that touch $T$ on both sides as shown in Figure~\ref{fig: width}. 
Then $2w(T)$ can be obtained by averaging the width of $T$ over all directions $g$ in $\R^n$. 

\begin{figure}[htp]			
  \centering \includegraphics[width=0.45\textwidth]{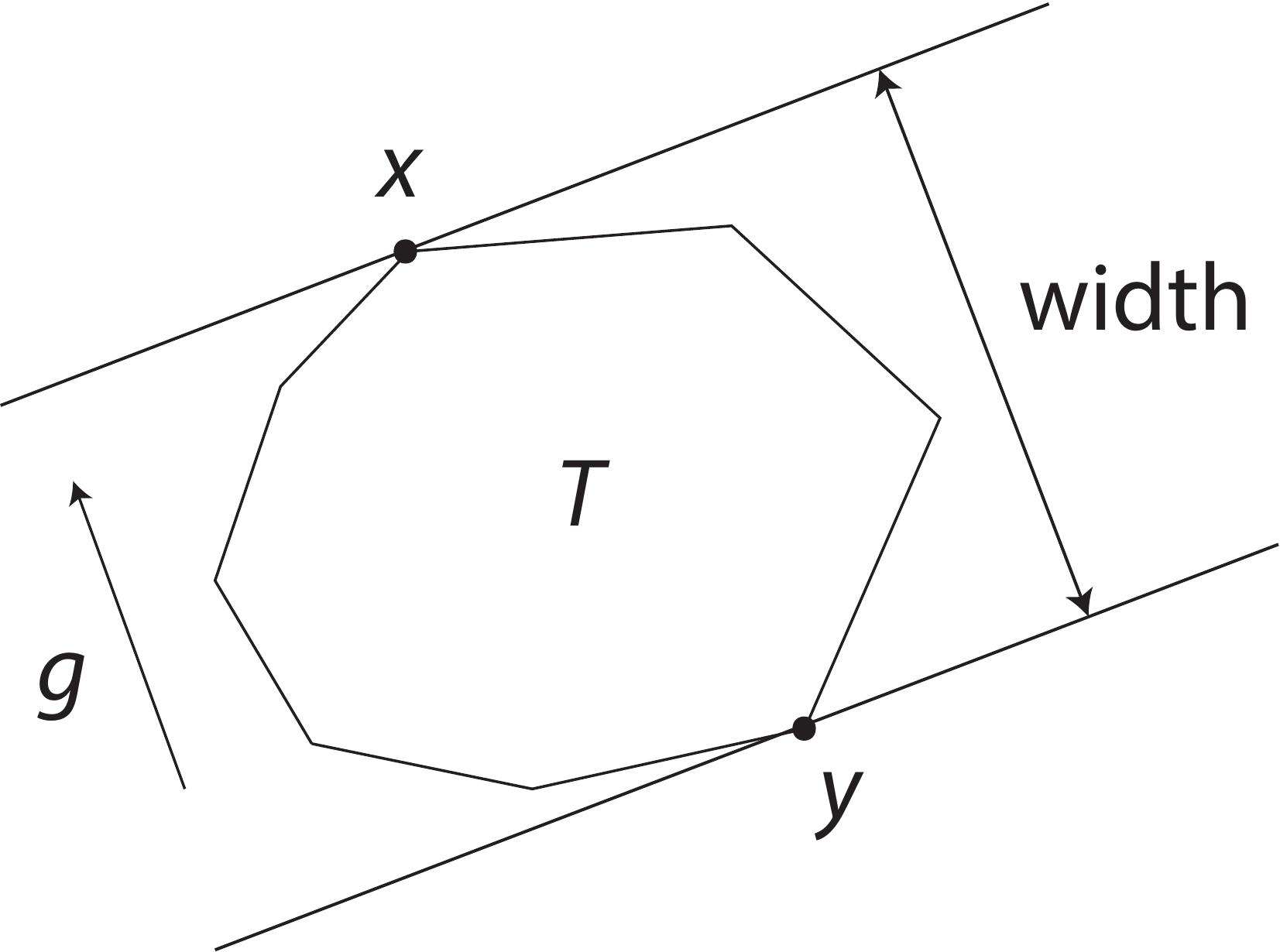} 
  \caption{The width of a set $T$ in the direction of $g$.}
  \label{fig: width}	
\end{figure}

This reasoning is valid except where we assumed that $g$ is a unit vector. Instead, for $g \sim N(0,I_n)$
we have $\E \|g\|_2^2 = n$ and 
$$
\|g\|_2 \approx \sqrt{n} \quad \text{with high probability.}
$$ 
(Check both these claims using Bernstein's inequality.)
Thus, we need to scale by the factor $\sqrt{n}$. Ultimately, the geometric interpretation of the 
Gaussian width becomes the following: {\em $w(T)$ is approximately $\sqrt{n}/2$ larger 
than the usual, geometric width of $T$ averaged over all directions}. 

A good exercise is to compute the Gaussian width and complexity for some simple sets, such as
the unit balls of the $\ell_p$ norms in $\R^n$, which we denote  
$B_p^n = \{x \in \R^n:\; \|x\|_p \le 1\}$.
In particular, we have 
\begin{equation}         \label{eq: width L2 L1}
\gamma(B_2^n) \sim \sqrt{n}, \quad \gamma(B_1^n) \sim \sqrt{\log n}.
\end{equation}
For any finite set $T \subset B_2^n$, we have 
\begin{equation}         \label{eq: width finite set}
\gamma(T) \lesssim \sqrt{\log |T|}.
\end{equation}
The same holds for Gaussian width $w(T)$. (Check these facts!)

A look a these examples reveals that the Gaussian width captures some non-obvious 
geometric qualities of sets. Of course, the fact that the Gaussian width of the unit 
Euclidean ball $B_2^n$ is or order $\sqrt{n}$ is not surprising: the usual, geometric 
width in all directions is $2$ and the Gaussian width is about $\sqrt{n}$ times that. 
But it may be surprising that the Gaussian width of the $\ell_1$ ball $B_1^n$ is much smaller, 
and so is the width of any finite set $T$ (unless the set has exponentially large cardinality).
As we will see later, Gaussian width nicely captures the geometric size of ``the bulk'' of a set. 

\subsection{Matrix deviation inequality}
%------------------

Now we are ready to answer the question we asked in the beginning of this lecture: 
what is the magnitude of the uniform deviation \eqref{eq: deviation question}?
The answer is surprisingly simple: it is bounded by the Gaussian complexity of $T$.
The proof is not too simple however, and we will skip it (see the notes after this lecture for references).

\begin{theorem}[Matrix deviation inequality]	\label{thm: matrix deviation}
  Let $A$ be an $m \times n$ matrix whose rows $A_i$ are independent, isotropic 
  and sub-gaussian random vectors in $\R^n$. 
  Let $T \subset \R^n$ be a fixed bounded set. Then
  $$
  \E \sup_{x \in T} \Big| \|Ax\|_2 - \sqrt{m} \|x\|_2 \Big| \le CK^2 \gamma(T)
  $$
  where $K = \max_i \|A_i\|_\psitwo$ is the maximal sub-gaussian norm\footnote{A definition 
    of the sub-gaussian norm of a random vector was given in Section~\ref{s: sub-gaussian vectors}.
    For example, if $A$ is a Gaussian random matrix with independent $N(0,1)$ entries, 
    then $K$ is an absolute constant.}
  of the rows of $A$. 
\end{theorem}

\begin{remark}[Tail bound] 
  It is often useful to have results that hold {\em with high probability} rather than in expectation. 
  There exists a high-probability version of the matrix deviation inequality, and it states the following. 
  Let $u \ge 0$. Then the event
  \begin{equation}         \label{eq: matrix deviation tail}
  \sup_{x \in T} \Big| \|Ax\|_2 - \sqrt{m} \|x\|_2 \Big| 
  \le CK^2 \left[ \gamma(T) + u \cdot \rad(T) \right]
  \end{equation}
  holds with probability at least $1-2\exp(-u^2)$. 
  Here $\rad(T)$ is the {\em radius} of $T$, defined as   
  $$
  \rad(T) := \sup_{x \in T} \|x\|_2.
  $$
  
  Since $\rad(T) \lesssim \gamma(T)$ (check!) 
  we can continue the bound \eqref{eq: matrix deviation tail} by
  $$
  \lesssim K^2 u \gamma(T)
  $$
  for all $u \ge 1$. This is a weaker but still a useful inequality. 
  For example, we can use it to bound all higher moments of the deviation: 
  \begin{equation}         \label{eq: deviation moments}
  \Big( \E \sup_{x \in T} \Big| \|Ax\|_2 - \sqrt{m} \|x\|_2 \Big|^p \Big)^{1/p} \le C_p K^2 \gamma(T)
  \end{equation}
  where $C_p \le C \sqrt{p}$ for $p \ge 1$. (Check this using Proposition~\ref{prop: sub-gaussian}.)
\end{remark}

\begin{remark}[Deviation of squares]
  It is sometimes helpful to bound the deviation of the {\em square} $\|Ax\|_2^2$ rather than $\|Ax\|_2$ itself.
  We can easily deduce the deviation of squares by using the identity
  $a^2-b^2= (a-b)^2 + 2b(a-b)$
  for $a = \|Ax\|_2$ and $b = \sqrt{m} \|x\|_2$. Doing this, we conclude that 
  \begin{equation}         \label{eq: MDI squares} 
  \E \sup_{x \in T} \Big| \|Ax\|_2^2 - m \|x\|_2^2 \Big| 
  \le CK^4 \gamma(T)^2 + CK^2 \sqrt{m} \rad(T) \gamma(T).
  \end{equation}
  (Do this calculation using \eqref{eq: deviation moments} for $p=2$.) 
  We will use this bound in Section~\ref{s: covariance sub-gaussian}.
\end{remark}

Matrix deviation inequality has many consequences. 
We will explore some of them now. 

\subsection{Deriving Johnson-Lindenstrauss Lemma}
%--------

We started this lecture by promising a result that is more general than Johnson-Lindenstrauss Lemma.
So let us show how to quickly derive Johnson-Lindenstrauss from the matrix deviation inequality. 
Theorem~\ref{thm: JL} from Theorem~\ref{thm: matrix deviation}.

Assume we are in the situation of the Johnson-Lindenstrauss Lemma (Theorem~\ref{thm: JL}).
Given a set $\XX \subset \R$, consider the normalized difference set 
$$
T := \Big\{ \frac{x-y}{\|x-y\|_2} :\; x, y \in \XX \text{ distinct vectors} \Big\}.
$$
Then $T$ is a finite subset of the unit sphere of $\R^n$, and thus \eqref{eq: width finite set} gives
$$
\gamma(T) \lesssim \sqrt{\log |T|} \le \sqrt{\log |\XX|^2} \lesssim \sqrt{\log |\XX|}.
$$
Matrix deviation inequality (Theorem~\ref{thm: matrix deviation}) then yields
$$
\sup_{x,y \in \XX} \Big| \frac{\|A(x-y)\|_2}{\|x-y\|_2} - \sqrt{m} \Big| 
\lesssim \sqrt{\log N} \le \e \sqrt{m}.
$$
with high probability, say $0.99$. (To pass from expectation to high probability,
we can use Markov's inequality. To get the last bound, we use the assumption on $m$ 
in Johnson-Lindenstrauss Lemma.)

Multiplying both sides by $\|x-y\|_2/\sqrt{m}$, we can write the last bound as follows. 
With probability at least $0.99$, we have 
$$
(1-\e) \|x-y\|_2 \le \frac{1}{\sqrt{m}} \|Ax - Ay\|_2 \le (1+\e) \|x-y\|_2 \quad \text{for all } x,y \in \XX.
$$
This is exactly the consequence of Johnson-Lindenstrauss lemma.

The argument based on matrix deviation inequality, which we just gave, 
can be easily extended for infinite sets. It allows one to state a version of 
Johnson-Lindenstrauss lemma for general, possibly infinite, sets, which depends
on the Gaussian complexity of $T$ rather than cardinality. (Try to do this!)

\subsection{Covariance estimation}				\label{s: covariance sub-gaussian}
%-------------------

In Section~\ref{s: covariance general}, we introduced the problem of covariance 
estimation, and we showed that 
$$
N \sim n \log n
$$
samples are enough to estimate the covariance matrix of a general distribution 
in $\R^n$. We will now show how to do better if the distribution is 
sub-gaussian. (Recall Section~\ref{s: sub-gaussian vectors} for the definition of sub-gaussian random vectors.)
In this case, we can get rid of the logarithmic oversampling and the boundedness
condition \eqref{eq: distribution bounded}. 

\begin{theorem}[Covariance estimation for sub-gaussian distributions]	\label{thm: covariance estimation sub-gaussian}
  Let $X$ be a random vector in $\R^n$ with covariance matrix $\Sigma$. 
  Suppose $X$ is sub-gaussian, and more specifically
  \begin{equation}         \label{eq: sub-gaussian normalized}
  \|\ip{X}{x}\|_\psitwo \lesssim \|\ip{X}{x}\|_{L^2} = \|\Sigma^{1/2} x\|_2 
  \quad \text{for any } x \in \R^n.
  \end{equation}
  Then, for every $N \ge 1$, we have
  $$
  \E \|\Sigma_N - \Sigma\| \lesssim \|\Sigma\| \; \Big( \sqrt{\frac{n}{N}} + \frac{n}{N} \Big).
  $$
\end{theorem}

This result implies that if, for $\e \in (0,1)$, we take a sample of size 
$$
N \sim \e^{-2} n,
$$
then we are guaranteed covariance estimation with a good relative error:
$$
\E \|\Sigma_N - \Sigma\| \le \e \|\Sigma\|.
$$

\begin{proof}
Since we are going to use Theorem~\ref{thm: matrix deviation}, we will need to first bring the random vectors 
$X, X_1, \ldots, X_N$ 
to the isotropic position. This can be done by a suitable linear transformation. You will easily 
check that there exists an {\em isotropic} random vector $Z$ such that 
$$
X = \Sigma^{1/2} Z.
$$
(For example, $\Sigma$ has full rank, set $Z := \Sigma^{-1/2} X$. Check the general case.)
Similarly, we can find independent and isotropic random vectors $Z_i$ such that  
$$
X_i = \Sigma^{1/2} Z_i, \quad i=1,\ldots,N.
$$
The sub-gaussian assumption \eqref{eq: sub-gaussian normalized} then implies that 
$$
\|Z\|_\psitwo \lesssim 1.
$$
(Check!)
Then 
$$
\|\Sigma_N - \Sigma\|
= \|\Sigma^{1/2} R_N \Sigma^{1/2}\|
\quad \text{where} \quad 
R_N := \frac{1}{N} \sum_{i=1}^N Z_i Z_i^\tran - I_n.
$$

The operator norm of a symmetric $n \times n$ matrix $A$ can be computed by maximizing 
the quadratic form over the unit sphere: $\|A\| = \max_{x \in S^{n-1}} |\ip{Ax}{x}|$. 
(To see this, recall that the operator norm is the biggest eigenvalue of $A$ in magnitude.)
Then 
$$
\|\Sigma_N - \Sigma\| = \max_{x \in S^{n-1}} \ip{\Sigma^{1/2} R_N \Sigma^{1/2} x}{x}
= \max_{x \in T} \ip{R_N x}{x}
$$
where $T$ is the ellipsoid 
$$
T := \Sigma^{1/2} S^{n-1}.
$$

Recalling the definition of $R_N$, we can rewrite this as
$$
\|\Sigma_N - \Sigma\| 
= \max_{x \in T} \Big| \frac{1}{N} \sum_{i=1}^N \ip{Z_i}{x}^2 - \|x\|_2^2 \Big|
= \frac{1}{N} \max_{x \in T} \big| \|Ax\|_2^2 - N \|x\|_2^2 \big|.
$$
Now we apply the matrix deviation inequality for squares \eqref{eq: MDI squares} and conclude that
$$
\|\Sigma_N - \Sigma\| 
\lesssim \frac{1}{N} \Big( \gamma(T)^2 + \sqrt{N} \rad(T) \gamma(T) \Big).
$$
(Do this calculation!)
The radius and Gaussian width of the ellipsoid $T$ are easy to compute:
$$
\rad(T) = \|\Sigma\|^{1/2} 
\quad \text{and} \quad 
\gamma(T) \le (\tr \Sigma)^{1/2}.
$$
Substituting, we get
$$
\|\Sigma_N - \Sigma\| 
\lesssim \frac{1}{N} \Big( \tr \Sigma + \sqrt{N \|\Sigma\| \tr \Sigma} \Big).
$$
To complete the proof, use that $\tr \Sigma \le n \|\Sigma\|$ (check this!) 
and simplify the bound. 
\end{proof}

\begin{remark}[Low-dimensional distributions]		\label{rem: low-dimensional}
  Similarly to Section~\ref{s: covariance general}, we can show that much fewer samples  
  are needed for covariance estimation of {\em low-dimensional} sub-gaussian distributions. 
  Indeed, the proof of Theorem~\ref{thm: covariance estimation sub-gaussian} actually yields
  \begin{equation}         \label{eq: low-dimensional}
  \E \|\Sigma_N - \Sigma\| \le \|\Sigma\| \; \Big( \sqrt{\frac{r}{N}} + \frac{r}{N} \Big)
  \end{equation}
  where 
  $$
  r = r(\Sigma^{1/2}) = \frac{\tr \Sigma}{\|\Sigma\|}
  $$
  is the {\em stable rank} of $\Sigma^{1/2}$.
  This means that covariance estimation is possible with 
  $$
  N \sim r
  $$
  samples. 
\end{remark}

\subsection{Underdetermined linear equations}		\label{s: underdetermined}
%-------------------

We will give one more application of the matrix deviation inequality -- this time, 
to the area of high dimensional inference. 
Suppose we need to solve a severely underdetermined system of linear equations: 
say, we have $m$ equations in $n \gg m$ variables. Let us write it in the matrix form as 
$$
y = Ax
$$
where $A$ is a given $m \times n$ matrix, $y \in \R^m$ is a given vector and 
$x \in \R^n$ is an unknown vector. We would like to compute $x$ from $A$ and $y$.

When the linear system is underdetermined, we can not find $x$ with any accuracy, unless we know
something extra about $x$. So, let us assume that we do have some a-priori information.
We can describe this situation mathematically by assuming that 
$$
x \in K
$$
where $K \subset \R^n$ is some known set in $\R^n$ that describes anything that we know 
about $x$ a-priori. (Admittedly, we are operating 
on a high level of generality here. If you need a concrete example, we will consider it 
in Section~\ref{s: sparse}.)

Summarizing, here is the problem we are trying to solve. Determine a solution $x = x(A,y,K)$ 
to the underdetermined linear equation $y=Ax$ as accurately as possible, assuming that
$x \in K$. 

A variety of approaches to this and similar problems were proposed during the last decade; see
the notes after this lecture for pointers to some literature.
The one we will describe here is based on optimization. 
To do this, it will be convenient to convert the {\em set} $K$ into a {\em function} on $\R^n$
which is called the {\em Minkowski functional} of $K$. This is basically a function 
whose level sets are multiples of $K$. To define it formally, assume that $K$ is {\em star-shaped},
which means that together with any point $x$, the set $K$ must contain the entire interval that connects 
$x$ with the origin; see Figure~\ref{fig: star-shaped} for illustration. The 
Minkowski functional of $K$ is defined as 
$$
\|x\|_K := \inf \big\{ t > 0:\; x/t \in K \big\}, \quad x \in \R^n.
$$
If the set $K$ is convex and symmetric about the origin, $\|x\|_K$ is actually a {\em norm} on $\R^n$. 
(Check this!)

\begin{figure}[htp]			
  \centering \includegraphics[width=0.5\textwidth]{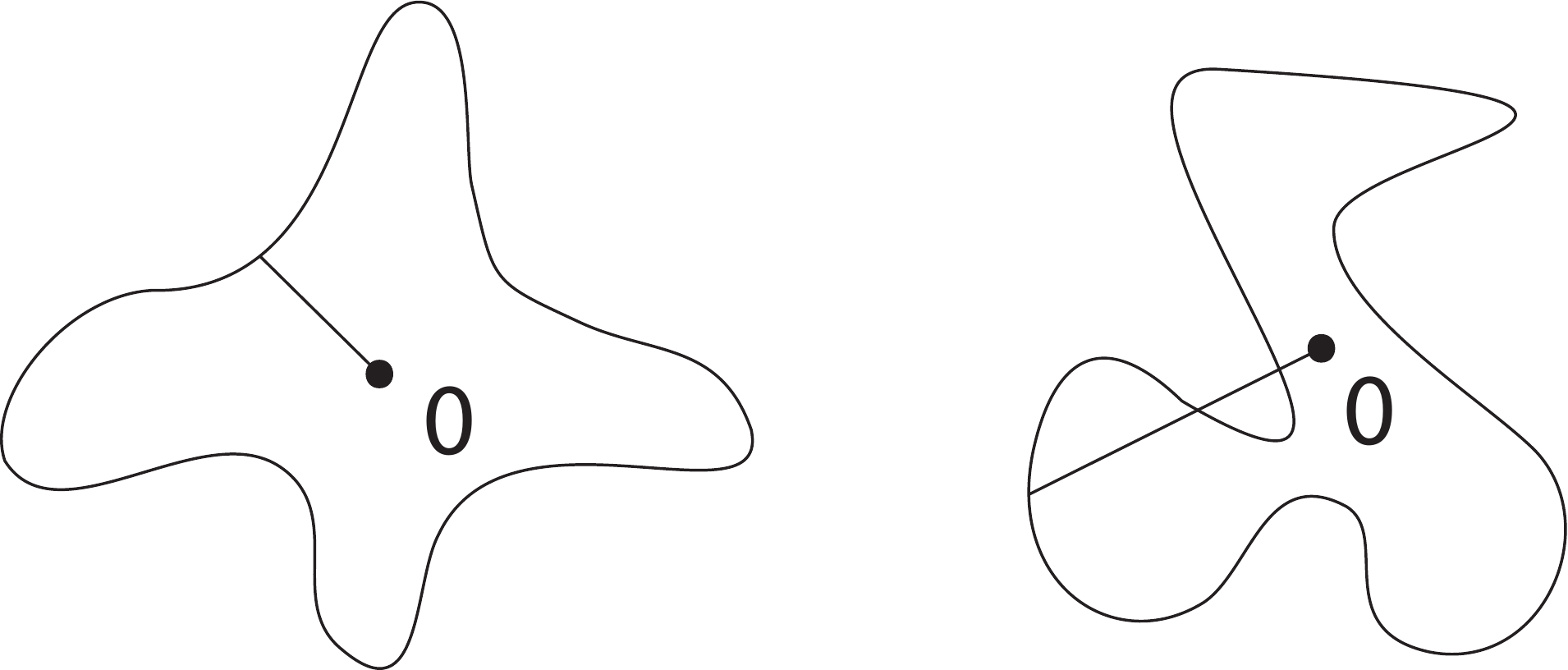} 
  \caption{The set on the left (whose boundary is shown) is star-shaped, the set on the right is not.}
  \label{fig: star-shaped}	
\end{figure}

Now we propose the following way to solve the recovery problem: solve the optimization program
\begin{equation}         \label{eq: opt}
\min \|x'\|_K \quad \text{subject to} \quad y = Ax'.
\end{equation}
Note that this is a very natural program: it looks at all solutions to the equation $y=Ax'$
and tries to ``shrink'' the solution $x'$ toward $K$. (This is what minimization of Minkowski functional is about.)

Also note that if $K$ is convex, 
this is a {\em convex optimization} program, and thus can be solved effectively 
by one of the many available numeric algorithms. 

The main question we should now be asking is -- would the solution to this program 
approximate the original vector $x$? The following result bounds the approximation error
for a {\em probabilistic model} of linear equations. Assume that $A$ is a random matrix as in 
Theorem~\ref{thm: matrix deviation}, i.e. $A$ is an $m \times n$ matrix whose rows $A_i$ are independent, isotropic and sub-gaussian random vectors in $\R^n$. 

\begin{theorem}[Recovery by optimization]			\label{thm: recovery}
  The solution $\hat{x}$ of the optimization program \eqref{eq: opt} satisfies\footnote{Here and in other similar results, 
  the notation $\lesssim$ will hide possible dependence on the sub-gaussian norms of the rows of $A$.}
  $$
  \E \|\hat{x}-x\|_2 \lesssim \frac{w(K)}{\sqrt{m}},
  $$
  where $w(K)$ is the Gaussian width of $K$.
\end{theorem}

\begin{proof}
Both the original vector $x$ and the solution $\hat{x}$  are feasible vectors for the optimization program \eqref{eq: opt}.
Then 
\begin{align*}
\|\hat{x}\|_K 
  &\le \|x\|_K \quad \text{(since $\hat{x}$ minimizes the Minkowski functional)} \\
  &\le 1 \quad \text{(since $x \in K$).}
\end{align*}
Thus both $\hat{x}, x \in K$.

We also know that $A\hat{x} = Ax = y$, which yields
\begin{equation}         \label{eq: kernel}
A(\hat{x}-x) = 0.
\end{equation}

Let us apply matrix deviation inequality (Theorem~\ref{thm: matrix deviation}) for $T := K-K$.
It gives
$$
\E \sup_{u,v \in K} \Big| \|A(u-v)\|_2 - \sqrt{m} \, \|u-v\|_2 \Big| \lesssim \gamma(K-K) = 2w(K),
$$
where we used \eqref{eq: width vs complexity} in the last identity.
Substitute $u = \hat{x}$ and $v = x$ here. We may do this since, as we noted above, 
both these vectors belong to $K$. But then the term $\|A(u-v)\|_2$ will equal zero 
by \eqref{eq: kernel}. It disappears from the bound, and we get 
$$
\E \sqrt{m} \, \|\hat{x}-x\|_2 \lesssim w(K).
$$
Dividing both sides by $\sqrt{m}$ we complete the proof.
\end{proof}

Theorem~\ref{thm: recovery} says that a signal $x \in K$
can be efficiently recovered from
$$
m \sim w(K)^2
$$
random linear measurements.

\subsection{Sparse recovery}		\label{s: sparse}
%----------------

Let us illustrate Theorem~\ref{thm: recovery} with an important specific example of the 
feasible set $K$. 

Suppose we know that the signal $x$ is {\em sparse}, 
which means that only a few coordinates of $x$ are nonzero. 
As before, our task is to recover $x$ from the random linear measurements given by the vector 
$$
y = Ax,
$$
where $A$ is an $m \times n$ random matrix. This is a basic example of {\em sparse recovery problems},
which are ubiquitous in various disciplines. 

The number of nonzero coefficients of a vector $x \in \R^n$, or the sparsity of $x$, is often denoted $\|x\|_0$. 
This is similar to the notation for the $\ell_p$ norm $\|x\|_p = (\sum_{i=1}^n |x_i|^p)^{1/p}$, 
and for a reason. You can quickly check that
\begin{equation}         \label{eq: L0 limit}
\|x\|_0 = \lim_{p \to 0} \|x\|_p 
\end{equation}
(Do this!) Keep in mind that neither $\|x\|_0$ nor $\|x\|_p$ for $0<p<1$ are actually {\em norms} on $\R^n$, 
since they fail triangle inequality. (Give an example.)

Let us go back to the sparse recovery problem. Our first attempt to recover $x$ is to try the following
optimization problem:
\begin{equation}         \label{eq: L0 optimization}
\min \|x'\|_0 \quad \text{subject to} \quad y = Ax'.
\end{equation}
This is sensible because this program selects the sparsest feasible solution. 
But there is an implementation caveat: the function $f(x) = \|x\|_0$ 
is highly non-convex and even discontinuous. There is simply no known algorithm to solve 
the optimization problem \eqref{eq: L0 optimization} efficiently. 

To overcome this difficulty, let us turn to the relation \eqref{eq: L0 limit} for an inspiration. 
What if we replace $\|x\|_0$ in the optimization problem \eqref{eq: L0 optimization}
by $\|x\|_p$ with $p>0$? The smallest $p$ for which $f(x) = \|x\|_p$ is a genuine norm (and thus 
a convex function on $\R^n$) is $p=1$. So let us try 
\begin{equation}         \label{eq: L1 optimization}
\min \|x'\|_1 \quad \text{subject to} \quad y = Ax'.
\end{equation}
This is a {\em convexification} of the non-convex program \eqref{eq: L0 optimization}, 
and a variety of numeric convex optimization methods are available to solve it efficiently. 

We will now show that $\ell_1$ minimization works nicely for sparse recovery. 
As before, we assume that $A$ is a random matrix as in Theorem~\ref{thm: matrix deviation}.

\begin{theorem}[Sparse recovery by optimization]			\label{thm: sparse recovery}
  Assume that an unknown vector $x \in \R^n$ has at most $s$ non-zero coordinates, 
  i.e. $\|x\|_0 \le s$.
  The solution $\hat{x}$ of the optimization program \eqref{eq: L1 optimization} satisfies
  $$
  \E \|\hat{x}-x\|_2 \lesssim \sqrt{\frac{s \log n}{m}} \; \|x\|_2.
  $$
\end{theorem}

\begin{proof}
Since $\|x\|_0 \le s$, Cauchy-Schwarz inequality shows that 
\begin{equation}         \label{eq: L1}
\|x\|_1 \le \sqrt{s} \, \|x\|_2.
\end{equation}
(Check!)
Denote the unit ball of the $\ell_1$ norm in $\R^n$ by $B_1^n$, i.e. 
$B_1^n := \{ x \in \R^n :\; \|x\|_1 \le 1\}$. Then we can rewrite \eqref{eq: L1} as the inclusion
$$
x \in \sqrt{s} \, \|x\|_2 \cdot B_1^n := K.
$$
Apply Theorem~\ref{thm: recovery} for this set $K$. We noted the Gaussian width of $B_1^n$ 
in \eqref{eq: width L2 L1}, so
$$
w(K) = \sqrt{s} \, \|x\|_2 \cdot w(B_1^n) 
\le \sqrt{s} \, \|x\|_2 \cdot \gamma(B_1^n) 
\le  \sqrt{s} \, \|x\|_2 \cdot \sqrt{\log n}.
$$
Substitute this in Theorem~\ref{thm: recovery} and complete the proof. 
\end{proof}

Theorem~\ref{thm: sparse recovery} says that an $s$-sparse signal $x \in \R^n$
can be efficiently recovered from
$$
m \sim s \log n
$$
random linear measurements.

\subsection{Notes}
%--------------------

For a more thorough introduction to Gaussian width and its role in high-dimensional estimation, 
refer to the tutorial \cite{V-estimation-tutorial} and the textbook \cite{V-textbook};
see also \cite{ALMT}. Related to Gaussian complexity is the notion of {\em Rademacher complexity} of $T$,
obtained by replacing the coordinates of $g$ by independent Rademacher (i.e. $\pm 1$ symmetric) random variables.
Rademacher complexity of classes of functions plays an important role in statistical learning theory,
see e.g. \cite{Mendelson} 

Matrix deviation inequality (Theorem~\ref{thm: matrix deviation}) is borrowed from \cite{LMPV}. 
In the special case where $A$ is a Gaussian random matrix, this result follows from the work of 
G.~Schechtman \cite{Schechtman} and could be traced back to results of Gordon \cite{Gordon85, Gordon87, Gordon88, Gordon92}.

In the general case of sub-gaussian distributions, 
earlier variants of Theorem~\ref{thm: matrix deviation} were proved by B.~Klartag and S.~Mendelson 
\cite{Klartag-Mendelson}, S.~Mendelson, A.~Pajor and N.~Tomczak-Jaegermann \cite{MPT}
and S.~Dirksen \cite{Dirksen}.

Theorem~\ref{thm: covariance estimation sub-gaussian} for covariance estimation 
can be proved alternatively using more elementary tools (Bernstein's inequality and $\e$-nets), see \cite{V-RMT-tutorial}. 
However, no known elementary approach exists for the {\em low-rank} covariance estimation
discussed in Remark~\ref{rem: low-dimensional}. The bound \eqref{eq: low-dimensional} 
was proved by V. Koltchinskii and K. Lounici \cite{KL} by a different method. 

In Section~\ref{s: underdetermined}, we scratched the surface of a recently developed 
area of {\em sparse signal recovery}, which is also called {\em compressed sensing}. 
Our presentation there essentially follows the tutorial \cite{V-estimation-tutorial}. 
Theorem~\ref{thm: sparse recovery} can be improved: if we take 
$$
m \gtrsim s \log (n/s)
$$
measurements, then with high probability 
the optimization program \eqref{eq: L1 optimization} recovers the unknown signal $x$ {\em exactly}, i.e. 
$$
\hat{x} = x.
$$
First results of this kind were proved by J.~Romberg, E.~Candes and T.~Tao \cite{CRT}
and a great number of further developments followed; refer e.g. to the book \cite{Foucart-Rauhut}
and the chapter in \cite{DDEK} for an introduction into this research area.

\section*{Acknowledgement}
I am grateful to the referees who made a number of useful suggestions, which led to better presentation of the material in this chapter.

%%%%%%%%%%%%%%%%%%%%%%%%%%%%%%%%%%%%%%%%%%%%%%%%%%%%%%%%%%%%%%%%%%%%%
%    
% To add references to your document, replace the two \bib commands below. 
%
%         1. You can use a list of \bib commands for the items you reference as is
%         done in our toy example here.
%
%         2. A second option is to use the command 
%             \bibselect{yourltbfile}
%         to point to a file of \bib commands that should be named 
%         yourltbfile.ltb and be placed in the same folder as your LaTeX
%         source files. 
%
%         3. A third option is to use the command 
%             \bibliography{yourbibfile}
%         to point to a file of BibTeX \bib commands that should be named 
%         yourltbfile.bbl and be placed in the same folder as your LaTeX
%         source files. 
%   
% If you use option 3. above, you should comment out or delete the lines
%            \begin{bibdiv}
%                \begin{biblist}
%        before the \bib command below as well as the line
%                  \end{biblist}
%              \end{bibdiv}
%        after it. 
%
% If you use options 2. or 3. and wish to make your source file self-contained you may
%         for final submission, simply copy the \bib entries to your \LaTeX\ file and
%         wrap them, if necessary, as indicated above.
%  
%%%%%%%%%%%%%%%%%%%%%%%%%%%%%%%%%%%%%%%%%%%%%%%%%%%%%%%%%%%%%%%%%%%%%

\bibspread

\end{document}